\documentclass[11pt,a4paper, reqno, centertags]{amsart}

\usepackage{indentfirst} 
\usepackage{amsmath, amsthm, amssymb,xfrac, mathrsfs} 
\usepackage{BOONDOX-calo, graphicx, accents}
\usepackage{tikz} 
\usetikzlibrary{hobby}
\usetikzlibrary{decorations.pathreplacing}
\usepackage{wrapfig} 
\usepackage{xcolor} 
\usepackage{verbatim} 
\usepackage{enumerate} 
\usepackage{fullpage}

\usepackage{hyperref} 
\hypersetup{
colorlinks=true,
citecolor=black!50!red,
linkcolor=black!50!red,
linktoc=all
}
\usepackage[all]{hypcap} 

\newcounter{constant}

\newcounter{bigconstant}

\newtheorem{theorem}{Theorem}[section]
\newtheorem{proposition}[theorem]{Proposition}
\newtheorem{lemma}[theorem]{Lemma}

\newtheorem{corollary}[theorem]{Corollary}

\numberwithin{equation}{section} 

\theoremstyle{definition}

\newtheorem{remark}[theorem]{Remark}

\newcommand{\R}{\mathbb{R}}
\newcommand{\C}{\mathbb{C}}
\newcommand{\T}{\mathbb{T}}

\newcommand{\N}{\mathbb{N}}

\makeatother

\begin{document}
    \title[On the lgebraic lower bound for the radius of analyticity]{On the algebraic lower bound for the radius of spatial analyticity for the Zakharov-Kuznetsov and modified Zakharov-Kuznetsov equations}
    \author[M. Baldasso and M. Panthee]{ Mikaela Baldasso and Mahendra Panthee}
    \address{Department of Mathematics, University of Campinas\\13083-859 Campinas, SP, Brazil}
\email{mikaelabaldasso@gmail.com, mpanthee@unicamp.br}

\maketitle

\begin{abstract}
    We consider the initial value problem (IVP) for the 2D generalized Zakharov-Kuznetsov (ZK) equation
\begin{eqnarray*}
\begin{cases}
\partial_{t}u+\partial_{x}\Delta u+\mu \partial_{x}u^{k+1}=0, \, &(x, y) \in \R^2, \, t \in \R,\\
u(x,  y,0)=u_0(x, y),
\end{cases}
\end{eqnarray*}
where  $\Delta=\partial_x^2+\partial_y^2$, $\mu=\pm 1$, $k=1,2$ and the initial data $u_0$ is real analytic in a strip around the $x$-axis of the complex plane and have radius of spatial analyticity $\sigma_0$.  For both $k=1$ and $k=2$, considering a symmetrized version, we prove that there exists $T_0>0$ such that the radius of spatial analyticity of the solution remains the same in the time interval $[-T_0,  T_0]$. We also consider the evolution of the radius of spatial analyticity when the local solution extends globally in time. For the Zakharov-Kuznetsov equation ($k=1$), we prove that, in both focusing ($\mu=1$) and defocusing ($\mu=-1$) cases, and for any $T> T_0$, the radius of analyticity cannot decay faster than $cT^{-4+\epsilon}$, $\epsilon>0$, $c>0$. For the modified Zakharov-Kuznetsov equation ($k=2)$ in the defocusing case ($\mu=-1$), we prove that the radius of spatial analyticity  cannot decay faster than $cT^{-\frac{4}{3}}$, $c>0$, for any $T>T_0$. These results on the algebraic lower bounds  for the evolution of the radius of analyticity improve the ones obtained by Shan and Zhang in \cite{SZ} and by Quian and Shan in \cite{QS} where the authors have obtained lower bounds involving exponential decay.
\end{abstract}



\section{Introduction}\label{s:intro}

We consider the initial value problem (IVP) with real analytic initial data for the generalized Zakharov-Kuznetsov (gZK) equation
\begin{eqnarray}\label{eq:gZK}
\begin{cases}
\partial_{t}u+\partial_{x}\Delta u+\mu \partial_{x}u^{k+1}=0, \, &(x,  y) \in \R^2, \, t \in \R,\\
u(x,   y, 0)=u_0(x,  y),
\end{cases}
\end{eqnarray}
where $\Delta=\partial_x^2+\partial_y^2$, $k \geq 1$ and the unknown $u(x,  y ,  t)$ is real-valued. When $k=1$, the equation \eqref{eq:gZK} is commonly referred as the Zakharov-Kuznetsov (ZK) equation, whereas, for $k=2$, it is called as the modified Zakharov-Kuznetsov (mZK) equation. These equations are extentions in two-dimensional space of the well known Korteweg de-Vries (KdV) and the modified KdV (mKdV) equations. The ZK equation was introduced by Zakharov and Kuznetsov  in \cite{ZK} to model the propagation of ion-acoustic waves in magnetic plasma in dimension 3. For a rigorous derivation of the ZK equation from the Euler–Poisson system with magnetic
field we refer to the work of Lannes, Linares and Saut in \cite{LLS}.

It is well-known that the generalized ZK equation possesses Hamiltonian structure and sufficiently smooth solution for the IVP \eqref{eq:gZK} enjoys the mass conservation
\begin{equation}\label{eq:mass conversation}
    M(u(t))=\int_{\R^2} u^2(x,  y,  t)dxdy
\end{equation}
and the energy conservation
\begin{equation}\label{eq:energy conservation}
    E(u(t))=\frac{1}{2}\int_{\R^2}\|\nabla u(x,  y,  t)\|^2dxdy -\frac{\mu}{k+2}\int_{\R^2}u^{k+2}(x,  y, t)dxdy.
\end{equation}

Well-posedness issues of the IVP \eqref{eq:gZK} with given initial data in the classical Sobolev spaces $H^s(\R^2)$ have been extensively studied by several authors. For the ZK equation, Faminskii \cite{F} proved the IVP \eqref{eq:gZK} is locally and globally well-posed in $H^m(\R^2)$, $m\geq 1$ integer, using a regularization technique. Later, Biagioni and Linares \cite{BL-03}  proved the local and global well-posedness in $H^1(\R^2)$ using smoothing estimates for the linear group. These results were improved by Linares and Pastor, in \cite{LP}, where the authors used some dispersive smoothing effects associated to the linear part of the ZK equation obtaining the local well-posedness in $H^s(\R^2)$ for $s>\frac{3}{4}$. By using the Fourier restriction norm method, Gr{\"u}nrock and Herr, in \cite{GH}, and Molinet and Pilod, in \cite{MP}, obtained the local well-posedness in $H^s(\R^2)$ for $s>\frac{1}{2}$. It is worth noticing that the authors in \cite{GH} introduced a linear transformation to obtain a symmetric symbol $\xi^3+\eta^3$ in the linear part of the equation. For this, they considered a linear change of variables $x \mapsto ax+by$ and $y \mapsto ax-by$ with $a=2^{-\frac{2}{3}}$ and $b=3^\frac{1}{2}2^{-\frac{2}{3}}$, so that \eqref{eq:gZK} can be rewritten as
\begin{eqnarray}\label{eq:new gZK}
\begin{cases}
\partial_{t}u+(\partial_{x}^3+\partial_{y}^3)u+\mu a(\partial_{x}+\partial_{y})u^{k+1}=0, \, &(x,   y) \in \R^2, \, t \in \R,\\
u(x,   y, 0)=u_0(x,  y).
\end{cases}
\end{eqnarray}
With this transformation, the linear part becomes symmetric facilitating the use of Bourgain's space framework, but there is a price to pay to deal with an extra derivative in the variable $y$ in the nonlinearity. The optimal local well-posedness result for given data in $H^s(\R^2)$ was obtained by Kinoshita, in \cite{K21}, for $s>-\frac{1}{4}$. The strategy of proof in \cite{K21} relies on the Fourier restriction norm method by proving bilinear estimates and using contraction mapping argument in the  $X^{s,b}$ space introduced by Bourgain \cite{B2}. As a corollary, using the principle of mass conservation, the same author also obtained the global well-posedness in $L^2(\R^2)$. This global result was recently improved for $s>-\frac{1}{13}$ in \cite{SWZ-23} using almost conserved quantity and  the {\em I-method} introduced by I-Team \cite{CKSTT-1, CKSTT-2}.  Finally, we refer \cite{HJ-20} for the local well-posedness result in  $H^{-\frac14}(\R^2)$.

Concerning the IVP associated to the mZK equation, in \cite{LP}, Linares and Pastor proved the local well-posedness in $H^s(\R^2)$ for $s>\frac{3}{4}$. Furthermore, they demonstrated that the problem is ill-posed if $s \leq 0$, in the sense that the data-to-solution map is not uniformly continuous, and hence well-posedness cannot be expected in the critical space $L^2(\R^2)$. After that, Ribaud and Vento \cite{RV}, improved this result for $s>\frac{1}{4}$. Recently, Bhattacharya et al. \cite{BFR} considered the symmetrized version of the mZK equation \eqref{eq:new gZK} and proved that this new equation preserves both mass conservation \eqref{eq:mass conversation} and, in the case $k=2$, modifies the energy conservation to become
\begin{equation}\label{eq:new energy conservation}
    E(u(t))=\frac{1}{2}\int_{\R^2}\|\nabla u(x, y,t)\|^2dxdy-\frac{1}{2}\int_{\R^2}(u_xu_y)(x, y, t)dxdy -\frac{\mu a}{4}\int_{\R^2}u^{4}(x, y,t)dxdy.
\end{equation}
Using this new form of the energy and the I-method, the authors in \cite{BFR} obtained the global well-posedness in $H^s(\R^2)$ for $s>\frac{3}{4}$ in the defocusing case. The best known local well-posedness result for the mZK equation was established by Kinoshita in \cite{K22} where the author extended the result obtained in \cite{RV} to include the index $s=\frac{1}{4}$.

Well-posednes issues and other properties of solutions to the IVP \eqref{eq:gZK}, considering several values of $k\geq 1$ and/or posed on domains other than $\R^2$ and $\R^3$ are extensively studied in the literature, see for example \cite{Br-00, FM-24, FLP-12, Gr-15, LP-11, LS-09, AM-24, AM-23, RT-08, Y-17} and references therein.

In this work we are interested in studying the well-posedness of the IVP \eqref{eq:gZK} for $k=1, 2$ with real analytic initial data $u_0$, i.e., initial data that are analytic in a strip  $$S_{\sigma} =\{x+iy\in\mathbb{C}^2:\, x,  y\in \R^2, \, |y_1|,|y_2|<\sigma\}$$ of $\C^2$ of width $\sigma$. 

For this purpose, we consider $u_0$ in the Gevrey space $G^{\sigma, s}(\R^{2})$, $\sigma>0$ and $s\in \R$, defined as the Banach space endowed with the norm
\begin{equation*}
  \|f\|_{G^{\sigma, s}}=\|e^{\sigma |\gamma|}\langle \gamma\rangle^{s}\widehat{f}(\gamma)\|_{L^{2}_{\gamma}}
\end{equation*}
where  $\gamma=(\xi,\eta)$ denotes the two dimensional spatial variable, $|\gamma|=|\xi|+|\eta|$, $\|\gamma\|=(\xi^2+\eta^2)^{\frac{1}{2}}$ and $\langle \gamma \rangle =(1+\|\gamma\|^2)^{\frac{1}{2}}$. Moreover, $\widehat{f}$ denotes the spatial Fourier transform of $f$,
\begin{equation*}
    \widehat{f}(\gamma)=\frac{1}{2\pi}\int_{\R^2}e^{-i(x\xi+y\eta)}f(x,y)dxdy.
\end{equation*}

For $\sigma=0$, the Gevrey space $G^{0, s}(\R^2)$ simply turns out to be the classical Sobolev space $H^s(\R^2)$. The interest in these spaces is due to the fact that for $\sigma>0$ and $s\in \R$, a function $f\in G^{\sigma, s}(\R^2)$ has a  holomorphic extension $F$ in the strip $S_{\sigma} $ defined above, see for example \cite{S2}. In this sense, $\sigma$ is called the uniform radius of analyticity of $f$.
A natural question concerning the IVP in this space is: given $u_0 \in G^{\sigma, s}$, is it possible to guarantee the existence of solution such that the radius of analyticity remains the same at least for short time? The other questions that naturally arise are: is it possible to extend the local solution to a larger time interval $[-T,T]$ for any $T>0$? And after extending the local solution globally in time, how does the radius of analyticity evolve in time? This sort of questions for the dispersive equations are widely studied in the literature, see for example \cite{BFH, BGK, FP, GK, H, K, SS, S, ST, MW-23} and references therein. See also \cite{HHP, HKS, HG} and references therein for the problems posed on the periodic domain $\T$.

As far as we know, the only results concerning the IVP \eqref{eq:gZK} with initial data in the Gevrey spaces $G^{\sigma,   s}(\R^2)$ are by Shan and Zhang in \cite{SZ} and by Quian and Shan in \cite{QS}. In both works the authors used the method introduced by Bona et al. in \cite{BGK} to obtain multilinear estimates in the Gevrey-Bourgain's spaces and proved the local well-posedness results. For the global well-posedness they followed approximation technique. More precisely, in the first work, \cite{SZ}, the authors proved that for $k\geq 2$ , $\sigma_0>0$ and $s>2$, if $u\in C([0,T],H^{s+1})$, $T \geq 1$, is a solution to the IVP \eqref{eq:gZK} with initial data $u_0 \in G^{\sigma_0,  s+1}(\R^2)$, then $u(x,y,t) \in C([0,T],G^{\frac{\sigma(T)}{2},  s}(\R^2))$, where
\begin{equation}\label{E-lb}
    \sigma(T)=\min\left\lbrace \sigma_0e^{-\delta_1}, cT^{-\frac{3k^2+(2s+8)k+2s+5}{3}}\right\rbrace
\end{equation}
with $\delta_1$ a constant determined by $\|u\|_{H^{ s+1}}$ and $\|u\|_{G^{\sigma_0,  s+1}}$. Note that the  results in \cite{SZ} do not include the ZK equation and the best lower bound for the radius of analyticity for the solutions to the mZK equation is
\begin{equation*}
    \sigma(T)=\min\left\lbrace \sigma_0e^{-\delta_1}, cT^{-(11+\epsilon)}\right\rbrace,
\end{equation*}
which can be inferred taking $k=2$ and $s=2+\epsilon$ in \eqref{E-lb}.

In the second work \cite{QS}, also using the technique from  Bona et al. in \cite{BGK}, the authors established the local well-posedness for the IVP~\eqref{eq:gZK} in the Gevrey space $G^{\sigma_0,  s}(\R^2)$  for $s>0$ when $k=1$ and for $s>1$ when $k\geq 2$. Moreover, they proved that the local solution extends globally in time for $s>2$ and the decay rate for the evolution of the radius of analyticity of the solutions is bounded below by $\sigma_0 e^{-\delta(t)}$, where
\begin{equation*}
    \delta(t)=\int_0^t\left(d_1+d_2\int_0^{t'}\|u(t'')\|_{H^{s+1}}^{k+2}dt''\right)^{k}dt',
\end{equation*}
with $d_1=\|u_0\|_{G^{\sigma_0,s+1}}^2$ and $d_2$ being a constant depending on $s$ and $p$.

Looking at the results explained above, two questions arise naturally.
\begin{enumerate}

\item[1.] Is it possible to obtain the local well-posedness results in $G^{\sigma, s}(\R^2)$ with smaller values of~$s$?
\item[2.] Is it possible to find a better decay rate for the evolution of the radius of analyticity when the local solution extends globally in time, in the sense that the radius decays slower than  the exponential rate?
\end{enumerate}

In this work, we will provide affirmative answers to the questions raised above for the  IVP~\eqref{eq:gZK} with initial data $u_0 \in G^{\sigma,s}(\R^2)$ for $k=1$ and $k=2$.  For this purpose, considering the symmetrized version of the gZK equation \eqref{eq:new gZK}, we will derive bilinear and trilinear estimates in the Gevrey- Bourgain's spaces and prove that, for short time, the solution remains analytic in the same initial strip when $s>-\frac{1}{4}$ for the ZK equation and when $s\geq \frac{1}{4}$ for the mZK equation. Moreover, we will construct almost conserved quantities at the $L^2$- and $H^1$-levels of Sobolev regularities (see \eqref{eq:conserved quantity ZK} and \eqref{eq:conserved quantity mZK} below) in order to extend the local solutions globally in time and to obtain algebraic lower bounds for the radius of analyticity.

Now, we state the main results of this work. Regarding the local well-posedness for the IVP associated to the ZK equation, we first prove the following  result.

\begin{theorem}\label{teo:local ZK}
    Let $\sigma>0$ and $s>-\frac{1}{4}$. For given $u_0 \in G^{\sigma,  s}(\R^2)$, there exists a time
    \begin{equation}\label{eq:T0 ZK}
T_0=\frac{c_0}{(1+\|u_0\|_{G^{\sigma,   s}}^2)^{d}}, \qquad c_0>0, \, d>1,
\end{equation}
 such that the IVP \eqref{eq:new gZK} with $k=1$ admits a unique solution $u \in C([-T_0, T_0]; G^{\sigma,  s}(\R^2))\cap X_{T_0}^{\sigma,  s,   \frac{1}{2}+\epsilon}$, for $0<\epsilon\ll1$ sufficiently small, satisfying
\begin{equation}\label{eq:bound for solution ZK teo}
\|u\|_{X_{T_0}^{\sigma,  s,   \frac{1}{2}+\epsilon}}\leq C\|u_0\|_{G^{\sigma,   s}},
\end{equation}
 where $X_{T_0}^{\sigma,  s,   \frac{1}{2}+\epsilon}$ is defined in \eqref{eq:GB time restriction}, Section \ref{sec:2}.
    
\end{theorem}

Concerning the IVP associated to the mZK equation, we prove the following local result.

\begin{theorem}\label{teo:local mZK}
    Let $\sigma>0$ and $s>\frac{1}{4}$. For given $u_0 \in G^{\sigma,  s}(\R^2)$, there exists 
    exists a time
    \begin{equation}\label{eq:T0 mZK}
T_0=\frac{c_0}{(1+\|u_0\|_{G^{\sigma,   s}}^2)^{d}}, \qquad c_0>0, \,  d>1,
\end{equation}
 such that the IVP \eqref{eq:new gZK} with $k=2$ admits a unique solution $u \in C([-T_0, T_0]; G^{\sigma,  s}(\R^2))\cap X_{T_0}^{\sigma,  s,   \frac{1}{2}+\epsilon}$, for $0<\epsilon\ll1$ sufficiently small, satisfying
\begin{equation}\label{eq:bound for solution mZK teo}
\|u\|_{X_{T_0}^{\sigma,  s,   \frac{1}{2}+\epsilon}}\leq C\|u_0\|_{G^{\sigma,   s}},
\end{equation}
where $X_{T_0}^{\sigma,  s,   \frac{1}{2}+\epsilon}$ is defined in \eqref{eq:GB time restriction}, Section \ref{sec:2}.
\end{theorem}

Our main results concerning the global well-posedness and the evolution of the radius of analyticity for the solutions are the following. For the IVP associated to the ZK equation  we prove the following global result that is valid for both focusing and defocusing cases.

\begin{theorem}\label{teo:global ZK}
    Let $\sigma_0>0$, $s>-\frac{1}{4}$, $u_0 \in G^{\sigma_0, s}(\R^2)$ and $u \in C([-T_0, T_0]; G^{\sigma_0,  s}(\R^2))$ be the local solution to the IVP \eqref{eq:new gZK} with $k=1$ given by Theorem \ref{teo:local ZK}. Then, for any $T\geq T_0$, the local solution $u$ extends globally in time satisfying
    \begin{equation*}
        u\in C([-T, T];G^{\sigma(T),  s}(\R^2)), \quad \text{with } \sigma(T) \geq \min \left\lbrace\sigma_0 , \, cT^{-4+\epsilon}\right\rbrace
    \end{equation*}
for any $\epsilon>0$, where $c$ is a positive constant depending on $s$, $\sigma_0$, $\|u_0\|_{G^{\sigma_0,s}}$ and $\epsilon$.
\end{theorem}

For the IVP associated to the mZK equation, we prove the following global result in the defocusing case.

\begin{theorem}\label{teo:global mZK}
    Let $\sigma_0>0$, $s\geq\frac{1}{4}$, $u_0 \in G^{\sigma_0, s}(\R^2)$ and $u \in C([-T_0, T_0]; G^{\sigma_0,  s}(\R^2))$ be the local solution to the IVP \eqref{eq:new gZK} with $k=2$ in the defocusing case ($\mu=-1$) given by Theorem \ref{teo:local mZK}. Then, for any $T\geq T_0$, the local solution $u$ extends globally in time satisfying
    \begin{equation*}
        u\in C([-T, T];G^{\sigma(T),  s}(\R^2)), \quad \text{with } \sigma(T) \geq \min \left\lbrace\sigma_0 , \, cT^{-\frac{4}{3}}\right\rbrace,
    \end{equation*}
where $c$ is a positive constant depending on $s$, $\sigma_0$ and $\|u_0\|_{G^{\sigma_0, s}}$.
\end{theorem}

As mentioned above, to prove Theorems \ref{teo:global ZK} and \ref{teo:global mZK}, we derive almost conserved quantities in $G^{\sigma, 0}(\R^2)$ and in $G^{\sigma, 1}(\R^2)$ spaces, respectively. For the ZK equation we construct an almost conserved quantity using the conservation law \eqref{eq:mass conversation} and a new bilinear estimate in Gevrey- Bourgain's space, see \eqref{eq:estimate partial derivative ZK} and \eqref{eq:conserved quantity ZK} below. While, for the mZK equation, in the defocusing case,  we construct an almost conserved quantity using the energy conservation in its modified form \eqref{eq:new energy conservation}, see \eqref{eq:conserved quantity mZK}.   Once having the almost conserved quantities at hand, using the general procedure from \cite{S2}, we are able to prove the global results by decomposing any interval of time $[0, T]$ into short subintervals and iterating the local results in each subinterval. During this iteration process there appears restrictions on the growth of the involved norms that provides the lower bound for the evolution of the radius of analyticity $\sigma(T)$.

This paper is organized as follows. In Section \ref{sec:2}, we define the Bourgain's spaces and record some preliminary estimates. The proofs of the local well-posedness results stated in Theorems \ref{teo:local ZK} and \ref{teo:local mZK} are contained in Section \ref{sec:3}. In Section \ref{sec:4}, we derive the almost conserved quantities and find the associated decay estimates. Finally, in Section \ref{sec:5}, we extend the local solutions globally in time and obtain algebraic lower bounds for the radius of analyticity stated in Theorems \ref{teo:global ZK} and \ref{teo:global mZK}.\\

\noindent \textbf{Notations:} 
Throughout this text, we will adopt notations commonly used in the context of partial differential equations. The two dimensional spatial variable pair will be denoted by $(x,   y)$ and its Fourier transform variable by $\gamma=( \xi,   \eta)$. As usual, we denote the time variable by $t$ and its Fourier transform variable by $\tau$. We will adopt the conventions $|\gamma|=|\xi|+|\eta|$, $\|\gamma\|=(\xi^2+\eta^2)^{\frac{1}{2}}$ and $\langle\gamma\rangle =(1+\|\gamma\|^2)^{\frac{1}{2}}$. The symbol $C$ represents various constants that may vary from one line to the next. We use $A \lesssim B$ to indicate an estimate of the form $A \leq cB$ and $A\sim B$ if $A\leq c_1B$ and $B\leq c_2A$.

\section{Function Spaces and multilinear estimates}\label{sec:2}
In this section we discuss the function spaces that will be used throughout this work and derive some multilinear estimates that play crucial role in the proofs.

First, regarding the Gevrey space defined in Section \ref{s:intro}, we have the embedding
\begin{equation}\label{eq:Gegrey embedding}
    G^{\sigma, s}\subset G^{\sigma',  s'} \quad \text{for all } 0<\sigma '<\sigma \text{ and } s, \, s' \in \R,
\end{equation}
and the inclusion is continuous in the sense that there exists a constant $C>0$ depending on $\sigma,\, \sigma ',\, s,\, s'$ such that
\begin{equation}\label{eq:Gevrey continuous embedding}
    \|f\|_{G^{\sigma ', s'}}\leq C\|f\|_{G^{\sigma, s}}.
\end{equation}

In addition to the Gevrey space, we use a space that is a mix between the Gevrey space and the Bourgain's space introduced in \cite{B1} and \cite{B2}. Given $\sigma\geq 0$ and $s,\, b \in \R$, we define the Gevrey-Bourgain's space, denoted by $X^{\sigma,  s, b}(\R^3)$, with the norm
\begin{equation*}
\|u\|_{X^{\sigma,  s, b}}=\|e^{\sigma |\gamma|}\langle \gamma \rangle^{s} \langle \tau - \xi^{3}-\eta^{3} \rangle^{b} \widehat{u}(\xi,  \eta,   \tau)\|_{L^{2}_{\tau,\xi,\eta}},
\end{equation*} 
where $\widehat{u}$ denotes the space-time Fourier transform of $u$. For $\sigma=0$, we recover the classical Bourgain's space  $X^{s, b}(\R^3)$ with the norm given by
\begin{equation*}
\|u\|_{X^{s, b}}=\|\langle \gamma \rangle^{s} \langle \tau - \xi^{3}-\eta^{3} \rangle ^{b} \widehat{u}(\xi,   \eta,  \tau)\|_{L^{2}_{\tau,\xi,\eta}}.
\end{equation*}

For $T>0$ the restrictions of $X^{ s, b}(\R^3)$ and $X^{\sigma,   s, b}(\R^3)$ to a time slab $\R^2 \times (-T,  T)$, denoted by $X_{T}^{ s, b}(\R^3)$ and $X_{T}^{\sigma,   s, b}(\R^3)$, respectively, are Banach spaces when equipped with the norms
\begin{align}
\|u\|_{X_{T}^{s, b}}&=\text{inf}\{\|v\|_{X^{s, b}}: v=u \text{ on } \R^2 \times (-T,  T)\},\label{eq:B time restrictin }\\
\|u\|_{X_{T}^{\sigma,  s, b}}&=\text{inf}\{\|v\|_{X^{\sigma,   s, b}}: v=u \text{ on } \R^2 \times (-T,  T)\}. \label{eq:GB time restriction}
\end{align}

To simplify the exposition we introduce the operator $e^{\sigma |D_{x,y}|}$ given by
\begin{equation*}
\widehat{e^{\sigma |D_{x,y}|}u}(\gamma)=e^{\sigma |\gamma|}\widehat{u}(\gamma)
\end{equation*}
so that, one has
\begin{align}
\|e^{\sigma |D_{x,y}|}u\|_{H^{s}}&=\|u\|_{G^{\sigma, s}}, \label{eq:exponential 0}\\
\|e^{\sigma |D_{x,y}|}u\|_{X^{s,  b}}&=\|u\|_{X^{\sigma, s,  b}}. \label{eq:exponential}
\end{align}
Substituting $u$ by $e^{\sigma |D_{x,y}|}u$, the relation \eqref{eq:exponential} allows us to carry out the properties of $X^{s,   b}$ and $X_T^{s,   b}$ spaces over $X^{\sigma, s,  b}$ and $X_T^{\sigma, s,  b}$ spaces.

Now we record some useful results that will be used in this  work. In the case $\sigma=0$, for the proof of the first lemma below we refer to Section 2.6 of \cite{T} and the second lemma follows by the argument used to prove Lemma 7 in \cite{S}. The proofs for $\sigma>0$ follows analogously using the relation \eqref{eq:exponential}.

\begin{lemma}Let $\sigma \geq 0$, $s \in \R$ and $b>\frac{1}{2}$. Then, $X^{\sigma, s,  b}(\R^3)\subset C(\R,  G^{\sigma,  s}(\R^2))$ and
\begin{equation*}
\sup_{t \in \R} \|u(t)\|_{G^{\sigma,   s}} \leq C \|u\|_{X^{\sigma, s,  b}},
\end{equation*}
where the constant $C>0$ depends only on $b$.
\end{lemma}

\begin{lemma}\label{lemma:restriction} Let $\sigma \geq 0$, $s \in \R$, $-\frac{1}{2}<b<\frac{1}{2}$ and $T>0$. Then, for any time interval $I \subset [-T,   T]$, we have 
\begin{equation*}
\|\chi_{I}u\|_{X^{\sigma, s,  b}} \leq C \|u\|_{X_{T}^{\sigma, s,  b}},
\end{equation*}
where $\chi_{I}$ is the characteristic function of $I$ and $C>0$ depends only on $b$.
\end{lemma}

Throughout this paper, $\psi \in C_0^{\infty}(\R)$ will denote a cut-off function such that $0\leq \psi(t) \leq 1$ and
\begin{equation}\label{eq:cut-off}
    \psi(t)=\begin{cases}
        1 &\text{ if } |t|\leq 1,\\
        0 &\text{ if } |t|\geq 2.\\
    \end{cases}
\end{equation}
Also, we define $\psi_T(t)=\psi\left(\frac{t}{T}\right)$ for $T>0$.

Consider the following IVP, for given $F(x,y,t)$ and $u_0(x,y)$,
\begin{eqnarray}\label{eq:Cauchy problem}
\begin{cases}
\partial_{t}u+(\partial_{x}^{3}+\partial_{y}^{3})u=F, \\
u(x,y,0)=u_0(x,y).
\end{cases}
\end{eqnarray}
Using the Duhamel's formula we may write the IVP \eqref{eq:Cauchy problem} in its equivalent integral equation form as
\begin{equation*}
u(t)=W(t)u_0-\int_0^t W(t-t') F(t')dt',
\end{equation*}
where $W(t)=e^{-t(\partial_{x}^{3}+\partial_{y}^{3})}=e^{it(D_x^{3}+D_y^{3})}$ is the semigroup associated to the linear problem. The semigroup $W(t)$ satisfies the following estimates in the $X^{\sigma, s,  b}$ spaces. For a detailed proof we refer to \cite{GTV} and \cite{QS}.

\begin{lemma}\label{lemma: semigroup estimates}Let $\sigma \geq 0$, $s \in \R$ and $\frac{1}{2}<b<b'<1$. Then, for all $0 <T\leq 1$, there is a constant $C=C(s,b)$ such that
\begin{equation}\label{eq:semigroup 1}
\|\psi(t)W(t)f(x, y)\|_{X^{\sigma, s,  b}} \leq C\|f\|_{G^{\sigma,   s}},
\end{equation}
and
\begin{equation}\label{eq:semigroup 2}
\left\| \psi_T(t) \int_{0}^{t}W(t-t')f(x, y,t')dt'\right\|_{X^{\sigma, s,  b}} \leq CT^{b'-b}\|f\|_{X^{\sigma, s,  b'-1}}.
\end{equation}
\end{lemma}

We also recall the following well-known classical inequality for the exponential function 
\begin{equation}\label{eq:exp function}
    e^x-1\leq x^{\alpha}e^x, \quad \forall \,   x\geq 0 \text{ and } \alpha \in [0,  1].
\end{equation}
The following result is a consequence of \eqref{eq:exp function}.
\begin{lemma}\label{lemma: inequality for e}
For $\sigma>0$, $\theta \in [0,1]$ and $x,   y \in \R^2$
\begin{equation*}
    e^{\sigma|x|}e^{\sigma|y|}-e^{\sigma|x+y|}\leq [2\sigma \min(|x|,|y|)]^{\theta}e^{\sigma|x|}e^{\sigma|y|}.
\end{equation*}    
\end{lemma}
\begin{proof}
In one dimensional case, i.e., for $x,y \in \R$, the proof can be found in \cite{SS}. 

For $x,y \in \R^2$, we will use the notations $x=(x_1,x_2)$, $y=(y_1,y_2)$. With these notations, we have to show
    \begin{equation}\label{eq: proof e}
        e^{\sigma(|x_1|+|x_2|+|y_1|+|y_2|)}-e^{\sigma(|x_1+y_1|+|x_2+y_2|)}\leq [2\sigma \min(|x_1|+|x_2|,|y_1|+|y_2|)]^{\theta}e^{\sigma(|x_1|+|x_2|+|y_1|+|y_2|)}.
    \end{equation}
    We separate the analysis in several cases depending on the signs of $x_1,x_2,y_1$ and $y_2$. Note that, it suffices to prove \eqref{eq: proof e} under the conditions
    \begin{itemize}
        \item $x_1$ and $y_1$ have the same signs and $x_2$ and $y_2$ have the same signs,
        \item $x_1\geq 0$, $y_1\leq 0$ and $x_2\geq 0$, $y_2\geq 0$,
        \item $x_1\geq 0$, $y_1\leq 0$ and $x_2\geq 0$, $y_2\leq 0$,
    \end{itemize}
since the other cases follow using symmetry of the norms involved.

\noindent {\bf Case 1:} \fbox{$x_1$ and $y_1$, and $x_2$ and $y_2$ have the same signs.} 
In this case, the left hand side of \eqref{eq: proof e} is equal to $0$ since $|x_1+y_1|=|x_1|+|y_1|$ and $|x_2+y_2|=|x_2|+|y_2|$ and the inequality is obvious.

\noindent {\bf Case 2:} \fbox{$x_1\geq 0$, $y_1\leq 0$ and $x_2\geq 0$, $y_2\geq 0$.}
In this case, one has $|x_2+y_2|=|x_2|+|y_2|$ and it is enough to show 
        \begin{equation}\label{eq: proof e 1}
            e^{\sigma(|x_1|+|y_1|)}-e^{\sigma|x_1+y_1|}\leq [2\sigma \min(|x_1|+|x_2|,|y_1|+|y_2|)]^{\theta}e^{\sigma(|x_1|+|y_1|)}.
        \end{equation}
        For this purpose, we separate the analysis in two sub-cases depending on the size of $x_1$ and $y_1$.

\noindent {\bf Sub-case 2.1:} \fbox{ $|y_1|\leq |x_1|$.} 
In this sub-case, one has $x_1+y_1\geq 0$ and using the estimate \eqref{eq:exp function}, the left side of \eqref{eq: proof e 1} becomes
            \begin{equation*}
                \begin{split}
                     e^{\sigma(|x_1|+|y_1|)}-e^{\sigma|x_1+y_1|}&=e^{\sigma(x_1-y_1)}-e^{\sigma(x_1+y_1)}\\
                     &=e^{\sigma(x_1+y_1)}(e^{-2\sigma y_1}-1)\\
                    &\leq e^{\sigma(x_1+y_1)}(2\sigma |y_1|)^{\theta}e^{-2\sigma y_1}\\
                    &=(2\sigma |y_1|)^{\theta}e^{\sigma(|x_1|+|y_1|)}.
                \end{split}
            \end{equation*}
             Since $|y_1|\leq |x_1|\leq |x|$ and $|y_1|\leq |y|$, one has
            \begin{equation*}
                 e^{\sigma(|x_1|+|y_1|)}-e^{\sigma|x_1+y_1|}\leq [2\sigma \min(|x|,|y|)]^{\theta}e^{\sigma(|x_1|+|y_1|)}.
            \end{equation*}

\noindent {\bf Sub-case 2.2:} \fbox{  $|y_1|\geq |x_1|$.}
In this sub-case, $-x_1\leq 0$, $-y_1\geq 0$ and $|-x_1|\leq |-y_1|$. Consequently, by symmetry of $|x_1|, |y_1|$ and $|x_1+y_1|$, using the \textbf{Sub-case 2.1}, one has
        \begin{equation*}
        \begin{split}
        e^{\sigma(|x_1|+|y_1|)}-e^{\sigma|x_1+y_1|}&=e^{\sigma(|-x_1|+|-y_1|)}-e^{\sigma|(-x_1)+(-y_1)|}\\
        &\leq [2\sigma \min(|x|,|y|)]^{\theta}e^{\sigma(|-x_1|+|-y_1|)}\\
        &=[2\sigma \min(|x|,|y|)]^{\theta}e^{\sigma(|x_1|+|y_1|)}.
        \end{split}
        \end{equation*}

\noindent {\bf Case 3:}  \fbox{ $x_1 \geq 0$, $y_1 \leq 0$ and $x_2\geq 0$, $y_2\leq 0$.} We separate the analysis of this case into two sub-cases depending on the size of $x_1, y_1, x_2$ and $y_2$.

\noindent {\bf Sub-case 3.1:} \fbox{$|y_2|\leq |x_2|$.} With this consideration, one has $x_2+y_2 \geq 0$ and we separate the analysis in two further sub-cases.

\noindent {\bf Sub-case 3.1.1:} \fbox{$|y_1|\leq|x_1|$.} In this case, one has $x_1+y_1\geq 0$ and using the estimate \eqref{eq:exp function}, the left side of \eqref{eq: proof e 1} becomes
                \begin{equation*}
                \begin{split}
                     e^{\sigma(|x_1|+|x_2|+|y_1|+|y_2|)}-e^{\sigma(|x_1+y_1|+|x_2+y_2|)}&=e^{\sigma(x_1+x_2-y_1-y_2)}-e^{\sigma(x_1+y_1+x_2+y_2)}\\
                     &=e^{\sigma(x_1+y_1+x_2+y_2)}(e^{-2\sigma(y_1+y_2)}-1)\\
                    &\leq e^{\sigma(x_1+y_1+x_2+y_2)}(2\sigma |y_1+y_2|)^{\theta}e^{-2\sigma( y_1+y_2)}\\
                    &\leq [2\sigma \min(|x|,|y|)]^{\theta}e^{\sigma(|x_1|+|x_2|+|y_1|+|y_2|)}.
                    \end{split}
                \end{equation*}

\noindent {\bf Sub-case 3.1.2:} \fbox{ $|y_1|\geq|x_1|$.} Since $-x_1\leq 0$, $-y_1\geq 0$ and $|-x_1|\leq|-y_1|$, the result follows by \textbf{Sub-case 3.1.1} by the argument used in \textbf{Sub-case 2.2}.

\noindent {\bf Sub-case 3.2:} \fbox{$|y_2|\geq |x_2|$.} 
One has $-x_2\leq 0$, $-y_2\geq 0$ and $|-x_2|\leq|-y_2|$ and the result follows from \textbf{Sub-case 3.1}.

\end{proof}
\vspace{0.2cm}
We recall the following well known Strichartz type estimate from \cite{BFR}
\begin{equation}\label{eq:strichartz 1 mZK}
\|u\|_{L_{t,x,y}^5} \leq C\|u\|_{X^{0,b}}, \text{ for all } b>\frac{1}{2}. 
\end{equation}
Moreover, for $p \in (5,\infty)$, we have the estimate
\begin{equation}\label{eq:strichartz 2 mZK}
\|u\|_{L_{t,x,y}^p} \leq C \|D^{\alpha(p)}u\|_{X^{0,b}}, \text{ for all } b>\frac{1}{2},
\end{equation}
where $\alpha(p)=(1+)\left(\frac{p-5}{p}\right)$ and $1+=1+\epsilon$.

The following result is immediate using the generalized Hölder inequality followed by \eqref{eq:strichartz 1 mZK} and \eqref{eq:strichartz 2 mZK} with $p=10$.
\begin{lemma} \label{lemma: L2 norm of product mZK}
For $b >\frac{1}{2}$, we have
\begin{equation*}
\|u_1 u_2 u_3\|_{L_{t,x,y}^2} \leq C \|u_1\|_{X^{0,b}}\|u_2\|_{X^{0,b}}\|u_{3}\|_{X^{\frac{1}{2}+,b}}. 
\end{equation*}
\end{lemma}

Now, we move to derive the bilinear and trilinear estimates that are key for obtaining the local well-posedness results and also the almost conserved quantities that are crucial in proving the global  results. 

For this purpose, we will use the Littlewood-Paley theory and  introduce an equivalent definition of  Bourgain's spaces in terms of dyadic decomposition. Let $N,L\geq 1$ be dyadic numbers, i.e., there exist $n_1, n_2 \in \N_0$ such that $N=2^{n_1}$ and $L=2^{n_2}$, and let $\psi \in C_0^{\infty}((-2,2))$ be an even cut-off function defined in \eqref{eq:cut-off}. Letting $\psi_1(t):=\psi(t)$ and $\psi_{N}(t):=\psi(tN^{-1})-\psi(2tN^{-1})$ for $N\geq 2$, the equality $\sum_{N}\psi_N(t)=1$ holds. Here we used $\sum_{N}=\sum_{N\in 2^{\N_0}}$. We define the frequency and modulation projections $P_N$ and $Q_L$ via Fourier transform by
\begin{equation*}
    \begin{split}
        \widehat{P_Nu}(\xi,\eta)&:=\psi_N(\|(\xi,\eta)\|)\widehat{u}(\xi,\eta,\tau)\\
        \widehat{Q_Lu}(\xi,\eta)&:=\psi_L(\tau-\xi^3-\eta^3)\widehat{u}(\xi,\eta,\tau).
    \end{split}
\end{equation*}
For $s,b \in \R$, we define the equivalent $X^{s,b}(\R^3)$ spaces with norm given by
\begin{equation*}
\|f\|_{X^{s,b}}=\left(\sum_{N,L}N^{2s}L^{2b}\|Q_LP_Nf\|_{L_{t,x,y}^2}^2\right)^{\frac{1}{2}}.
\end{equation*}

In this setting, we recall the following Strichartz estimate from \cite{K21}.

\begin{lemma}\label{lemma: Strichartz Q_L} For $p\geq 4$ and $\frac{2}{p}+\frac{2}{q}=1$,
\begin{equation*}
    \|Q_Lu\|_{L_t^pL_{x,y}^q}\leq C L^{\frac{2}{3p}+\frac{1}{q}}\|Q_Lu\|_{L_{t,x,y}^2}.
\end{equation*}   
\end{lemma}

In what follows, we record the bilinear estimate obtained by Kinoshita in \cite{K21} which plays crucial role to establish the local well posedness for the IVP \eqref{eq:gZK} with $k=1$ and initial data in $H^s(\R^2)$ for $s>-\frac{1}{4}$.

\begin{lemma}\label{lemma: Kinoshita} For any $s>-\frac{1}{4}$, there exist $b\in \left(\frac{1}{2},1\right)$, $\epsilon>0$ and $C>0$ such that
 \begin{equation}\label{eq:estimate partial derivative kinoshita}
\|(\partial_x+\partial_y)(uv)\|_{X^{s,b-1+\epsilon}} \leq C\|u\|_{X^{s,b}}\|v\|_{X^{s,b}}.
\end{equation}
\end{lemma}

\begin{remark}\label{rem-bilinear}
Note that from Lemma  \ref{lemma: Kinoshita} one infers that,  given  $s>-\frac{1}{4}$, there exist $b=b(s)\in \left(\frac{1}{2},1\right)$, $\epsilon=\epsilon(s)>0$ and $C=C(s)>0$ depending on $s$ such that the estimate \eqref{eq:estimate partial derivative kinoshita} holds. The estimate \eqref{eq:estimate partial derivative kinoshita} in this form is sufficient to prove the local well-posedness for the IVP associated to the ZK equation for given data in  $H^s(\R^2)$  and also, with usual adaptation in the Gevrey-Bourgain's space, for data in $G^{\sigma,  s}(\R^2)$ for $s > -\frac{1}{4}$. However, to control the growth of the almost conserved quantity in $G^{\sigma,  0}(\R^2)$ that we will introduce in Section \ref{sec:4} we need a more refined form of the estimate \eqref{eq:estimate partial derivative kinoshita}. More precisely, we need to guarantee that for the same $b=b(s)$ and $\epsilon=\epsilon(s)$ of Lemma \ref{lemma: Kinoshita} one has  $b(s)\leq b(0)$. The necessity of this refined version is explained in  Remark \ref{remark:faixa} below.
\end{remark}

In sequel, we state and prove a refined version of the bilinear estimate that is crucial to obtain an almost conserved quantity in $G^{\sigma,  0}(\R^2)$ space.
\begin{lemma} \label{lemma:estimates partial ZK} Let $s>-\frac{1}{4}$. Then, for $\epsilon (s)=\min(\frac{1}{24}, \frac{s}{6}+\frac{1}{24}) $, we have
\begin{equation}\label{eq:estimate partial derivative ZK}
\|(\partial_x+\partial_y)(u_1u_2)\|_{X^{s,-\frac{1}{2}+2\epsilon(s)}} \leq C\|u_1\|_{X^{s,\frac{1}{2}+\epsilon(s)}}\|u_2\|_{X^{s,\frac{1}{2}+\epsilon(s)}},
\end{equation}
where $C>0$ depends on $s$. In particular, for any $s>-\frac{1}{4}$, one has $\epsilon(s)\leq\epsilon(0)$.
\end{lemma}

\begin{proof} The proof of this lemma follows using the idea of proof of  \eqref{eq:estimate partial derivative kinoshita} presented in \cite{K21} (see Theorem~$2.1$ there). For the sake of completeness, we  provide a proof for the estimate \eqref{eq:estimate partial derivative ZK} introducing several details needed in the proof presented in \cite{K21}. 

Using duality followed by dyadic decomposition, to obtain \eqref{eq:estimate partial derivative ZK} it suffices to show
\begin{equation}\label{eq:3.2}
\begin{split}
    &\sum_{N_j,L_j(j=0,1,2)} \left|\int ((\partial_x+\partial_y)(Q_{L_0}P_{N_0}u_0))(Q_{L_1}P_{N_1}u_1)(Q_{L_2}P_{N_2}u_2)dtdxdy\right| \leq \\
    & \hspace{6.5cm} \leq C\|u_0\|_{X^{-s,\frac{1}{2}-2\epsilon(s)}}\|u_1\|_{X^{s,\frac{1}{2}+\epsilon(s)}}\|u_2\|_{X^{s,\frac{1}{2}+\epsilon(s)}}.
    \end{split}
\end{equation}
For simplicity, we use the notations
\begin{equation*}
    \begin{split}
        &L_{\max}=\max(L_0,L_1,L_2), \quad N_{\max}=\max(N_0,N_1,N_2), \quad N_{\min}=\min(N_0,N_1,N_2),\\
        &u_{N_i,L_i}=Q_{L_i}P_{N_i}u_i.
    \end{split}
\end{equation*}

Using Plancherel's Theorem, one can see that \eqref{eq:3.2} is verified by showing
\begin{equation}\label{eq:3.2``}
\begin{split}
    &\left|\int_{*} |\xi+\eta|\widehat{u}_{N_0,L_0}(\xi,\eta,\tau)\widehat{u}_{N_1,L_1}(\xi_1,\eta_1,\tau_1)\widehat{u}_{N_2,L_2}(\xi_2,\eta_2,\tau_2)d\sigma_1 d\sigma_2 \right| \leq\\
    & \qquad \qquad   \leq C\frac{N_1^sN_2^s}{N_0^s}L_0^{\frac{1}{2}-2\epsilon(s)}(L_1L_2)^{\frac{1}{2}+\epsilon(s)}\|u_{N_0,L_0}\|_{L_{t,x,y}^{2}}\|u_{N_1,L_1}\|_{L_{t,x,y}^{2}}\|u_{N_2,L_2}\|_{L_{t,x,y}^{2}},
    \end{split}
\end{equation}
where $d\sigma_j=d\tau_jd\xi_jd\eta_j$ and $\int_*$ denotes the integral over the set $(\xi,\eta,\tau)=(\xi_1+\xi_2,\eta_1+\eta_2,\tau_1+\tau_2)$. 

In a similar way, another alternative to prove \eqref{eq:3.2} consists of verifying that
\begin{equation}\label{eq:3.3``}
\begin{split}
    &\left|\int_{**} |\xi+\eta|\widehat{u}_{N_0,L_0}(\xi,\eta,\tau)\widehat{u}_{N_1,L_1}(\xi_1,\eta_1,\tau_1)\widehat{u}_{N_2,L_2}(\xi_2,\eta_2,\tau_2)d\sigma_1 d\sigma \right| \leq\\
    & \qquad \qquad   \leq C\frac{N_1^sN_2^s}{N_0^s}L_0^{\frac{1}{2}-2\epsilon(s)}(L_1L_2)^{\frac{1}{2}+\epsilon(s)}\|u_{N_0,L_0}\|_{L_{t,x,y}^{2}}\|u_{N_1,L_1}\|_{L_{t,x,y}^{2}}\|u_{N_2,L_2}\|_{L_{t,x,y}^{2}},
    \end{split}
\end{equation}
where $d\sigma=d\tau d\xi d\eta$, $d\sigma_1=d\tau_1d\xi_1d\eta_1$ and $\int_{**}$ denotes the integral over the set \break $(\xi_2,\eta_2,\tau_2)=(\xi_1+\xi,\eta_1+\eta, \tau_1+\tau)$. 

Moreover, by Plancherel's Theorem, it follows that \eqref{eq:3.2``} and \eqref{eq:3.3``} are verified by showing
\begin{equation}\label{eq:3.3}
\begin{split}
    &N_0 \left|\int u_{N_0,L_0}u_{N_1,L_1}u_{N_2,L_2}dtdxdy\right|\leq\\
    & \qquad \qquad   \leq C\frac{N_1^sN_2^s}{N_0^s}L_0^{\frac{1}{2}-2\epsilon(s)}(L_1L_2)^{\frac{1}{2}+\epsilon(s)}\|u_{N_0,L_0}\|_{L_{t,x,y}^{2}}\|u_{N_1,L_1}\|_{L_{t,x,y}^{2}}\|u_{N_2,L_2}\|_{L_{t,x,y}^{2}}.
    \end{split}
\end{equation}

So, in what follows, we will show \eqref{eq:3.2``}, \eqref{eq:3.3``} or \eqref{eq:3.3}, depending on the case undertaken to get \eqref{eq:3.2}.

If $N_0 \sim N_1 \sim N_2\sim 1$, the result follows easily. In fact, using the Strichartz estimates from Lemma \ref{lemma: Strichartz Q_L} with $p=q=4$, one has
 \begin{equation}\label{eq:trivial case}
 \begin{split}
     \!\!\!\!\!   \left|\int u_{N_0,L_0}u_{N_1,L_1}u_{N_2,L_2}dtdxdy\right|& \leq \|u_{N_0,L_0}\|_{L_{t,x,y}^2}\|u_{N_1,L_1}\|_{L_{t,x,y}^{4}}\|u_{N_2,L_2}\|_{L_{t,x,y}^{4}}\\
        &\leq C(L_1L_2)^{\frac{5}{12}}\|u_{N_0,L_0}\|_{L_{t,x,y}^{2}}\|u_{N_1,L_1}\|_{L_{t,x,y}^{2}}\|u_{N_2,L_2}\|_{L_{t,x,y}^{2}}\\
        &\leq C(L_0L_1L_2)^{\frac{5}{12}}\|u_{N_0,L_0}\|_{L_{t,x,y}^{2}}\|u_{N_1,L_1}\|_{L_{t,x,y}^{2}}\|u_{N_2,L_2}\|_{L_{t,x,y}^{2}}.
        \end{split}
    \end{equation}
Consequently, for every $0<\epsilon\leq \frac{1}{24}$, one has
\begin{equation*}
\begin{split}
    \left|\int u_{N_0,L_0}u_{N_1,L_1}u_{N_2,L_2}dtdxdy\right|&\leq CL_0^{\frac{1}{2}-2\epsilon}(L_1L_2)^{\frac{1}{2}+\epsilon}\|u_{N_0,L_0}\|_{L_{t,x,y}^{2}}\|u_{N_1,L_1}\|_{L_{t,x,y}^{2}}\|u_{N_2,L_2}\|_{L_{t,x,y}^{2}}
    \end{split}
\end{equation*}    
and \eqref{eq:3.3} is proved.

Henceforth, we assume $1 \ll N_{\max}$. Under this condition, the proof of \eqref{eq:3.2} is divided in the follwing cases:
\begin{itemize}
    \item \textbf{High modulation:} $L_{\max}\geq C(N_{\max})^3$.
    
    \item \textbf{Low modulation:} $L_{\max}\ll(N_{\max})^3$.
    \begin{itemize}
        \item \textit{Non-parallel interactions:}
            \begin{itemize}
        \item[(i)] $N_{\max} \leq 2^{22}N_{\min}$,
        \item[(ii)] $|\sin\angle{((\xi_1,\eta_1),(\xi_2,\eta_2))}|\geq 2^{-22}$,
    \end{itemize}
        \item \textbf{Parallel interactions:}\\
     $\left\{ \begin{array}{lll}
\text{If } N_{\min}=N_0,  \, |\sin\angle{((\xi_1,\eta_1),(\xi_2,\eta_2))}|\geq 2^{-20}, \\
\text{If } N_{\min}=N_1,  \, |\sin\angle{((\xi,\eta),(\xi_2,\eta_2))}|\geq 2^{-20},\\
\text{If } N_{\min}=N_2,  \, |\sin\angle{((\xi,\eta),(\xi_1,\eta_1))}|\geq 2^{-20}.
\end{array}\right.$
\end{itemize}
where $\angle{((\xi_i,\eta_i),(\xi_j,\eta_j))}\in [0,\pi]$ is the angle between $(\xi_i,\eta_i)$ and $(\xi_j,\eta_j)$.
\end{itemize}

\vspace{0.2cm}
In sequel, we provide proof of \eqref{eq:3.2} considering the cases described above.
\vspace{0.2cm}

\noindent {\bf Case 1 (High modulation):} \fbox{$L_{\max}\geq C(N_{\max})^3$.}
 In this case, we will show \eqref{eq:3.3}. First note that, under the condition of this case, from Proposition 3.2 of \cite{K21}, we have 
    \begin{equation}\label{eq:2.10}
        \left|\int u_{N_0,L_0}u_{N_1,L_1}u_{N_2,L_2}dtdxdy\right| \leq C(N_{\max})^{-\frac{5}{4}}(L_0L_1L_2)^{\frac{5}{12}}\|u_{N_0,L_0}\|_{L_{t,x,y}^{2}}\|u_{N_1,L_1}\|_{L_{t,x,y}^{2}}\|u_{N_2,L_2}\|_{L_{t,x,y}^{2}}.
    \end{equation}
    From \eqref{eq:2.10}, one has that for every $0<\epsilon\leq \frac{1}{24}$
        \begin{equation}\label{eq:2.10'}
        \left|\int u_{N_0,L_0}u_{N_1,L_1}u_{N_2,L_2}dtdxdy\right| \leq C(N_{\max})^{-\frac{5}{4}}L_0^{\frac{1}{2}-2\epsilon}(L_1L_2)^{\frac{1}{2}+\epsilon}\|u_{N_0,L_0}\|_{L_{t,x,y}^{2}}\|u_{N_1,L_1}\|_{L_{t,x,y}^{2}}\|u_{N_2,L_2}\|_{L_{t,x,y}^{2}}.
    \end{equation}
    Consequently, \eqref{eq:3.3} follows from \eqref{eq:2.10'} if we guarantee that for any $s>-\frac{1}{4}$,
    \begin{equation}\label{eq:2.11}
        \frac{N_0}{(N_{\max})^{\frac{5}{4}}} \leq CN_0^{-s}N_1^s N_2^s. 
    \end{equation}
We divide the proof of \eqref{eq:2.11} in two different sub-cases, $s\geq 0$ and $-\frac{1}{4}<s<0$.\\

\noindent{\bf Sub-case 1.1:}  \fbox{$s \geq 0$.}  We further divide this sub-case to two different sub-cases.\\

\noindent{\bf Sub-case 1.1.1:}  \fbox{ $N_{max}=N_0$.} In this sub-case, without loss of generality, one can suppose that $N_0\sim N_1$ and consequently
\begin{equation*}
    \frac{N_0}{(N_{\max})^{\frac{5}{4}}}\leq 1 \sim \frac{N_1^s}{N_0^{s}} \leq CN_1^s N_2^sN_0^{-s}.
\end{equation*}

\noindent{\bf Sub-case 1.1.2:}  \fbox{ $N_{\max}=N_1$.} In this sub-case, one simply has

            \begin{equation*}
    \frac{N_0}{(N_{\max})^{\frac{5}{4}}} \leq 1\leq \frac{N_1^s}{N_0^{s}} \leq N_1^sN_2^sN_0^{-s}.
            \end{equation*}

\noindent{\bf Sub-case 1.2:}  \fbox{ $-\frac{1}{4}<s<0$.} We analyse this sub-case considering two different situations.\\

\noindent{\bf Sub-case 1.2.1:}  \fbox{ $N_{\max}=N_0$.} In this case, without loss of generality, one can assume that $N_1 \leq N_2$. Since $-\frac{1}{4}<s<0$, one has $\frac{1}{4}+s>0$, and consequently
\begin{equation*}
    \frac{N_0}{N_1^sN_2^sN_0^{-s}(N_{\max})^{\frac{5}{4}}} =  \frac{1}{N_1^sN_2^sN_0^{\frac{1}{4}-s}}\leq  \frac{1}{N_0^{\frac{1}{4}+s}} \leq 1.
\end{equation*}
So, we conclude that
\begin{equation*}
     \frac{N_0}{(N_{\max})^{\frac{5}{4}}} \leq CN_0^{-s}N_1^s N_2^s.
\end{equation*}

\noindent{\bf Sub-case 1.2.2:}  \fbox{ $N_{\max}=N_1$.}  In this case, one has to show
    \begin{equation}\label{eq:N2>N0}
        \frac{N_0}{N_1^{\frac{5}{4}}} \leq CN_1^sN_2^sN_0^{-s}.
    \end{equation}
    The estimate \eqref{eq:N2>N0} is equivalent to
    \begin{equation*}
        1\leq CN_1^{s+\frac{5}{4}}N_2^sN_0^{-s-1}.
    \end{equation*}
    The last inequality is true since $s+\frac{1}{4}>0$ and consequently
    \begin{equation*}
    \begin{split}
    N_1^{s+\frac{5}{4}}N_2^sN_0^{-s-1}&=N_1^{\frac{1}{4}}N_1^{s+1}N_2^sN_0^{-s-1}
    \geq N_2^{\frac{1}{4}}N_0^{s+1}N_2^sN_0^{-s-1}
    =N_2^{s+\frac{1}{4}}
    \geq 1.
    \end{split}
    \end{equation*}

\noindent{\bf Case 2 (Low modulation):} \fbox{ $L_{\max}\ll N_{\max}^3$.}
We divide the analysis in two different sub-cases.\\

\noindent{\bf Sub-case 2.1 (Non-parallel interactions):} In this case, we suppose  
    \begin{itemize}
        \item[(i)] $L_{\max}\leq 2^{-100}(N_{\max})^3$,
        \item[(ii)] $N_{\max} \leq 2^{22}N_{\min}$,
        \item[(iii)] $|\sin\angle{((\xi_1,\eta_1),(\xi_2,\eta_2))}|\geq 2^{-22}$,
    \end{itemize}
where $\angle{((\xi_1,\eta_1),(\xi_2,\eta_2))}\in [0,\pi]$ is the angle between $(\xi_1,\eta_1)$ and $(\xi_2,\eta_2)$. Under the conditions (i), (ii) and (iii), it is shown in \cite{K21} (see equation (3.5) in page 455 there), that
    \begin{equation}\label{eq:2.17}
    \begin{split}
        &\left|\int_{*} \widehat{u}_{N_0,L_0}(\xi,\eta,\tau)\widehat{u}_{N_1,L_1}(\xi_1,\eta_1,\tau_1)\widehat{u}_{N_2,L_2}(\xi_2,\eta_2,\tau_2)d\sigma_1d\sigma_2\right| \\
        & \qquad \qquad \qquad \qquad \leq C(N_{\max})^{-\frac{5}{4}}(L_0L_1L_2)^{\frac{5}{12}}\|\widehat{u}_{N_0,L_0}\|_{L_{\tau,\xi,\eta}^{2}}\|\widehat{u}_{N_1,L_1}\|_{L_{\tau,\xi,\eta}^{2}}\|\widehat{u}_{N_2,L_2}\|_{L_{\tau,\xi,\eta}^{2}},
        \end{split}
    \end{equation}
    where $d\sigma_j=d\tau_jd\xi_jd\eta_j$ and $\int_*$ denotes the integral over the set $(\xi,\eta,\tau)=(\xi_1+\xi_2,\eta_1+\eta_2,\tau_1+\tau_2)$. The desired result follows exactly as in the {\bf Case 1} since, using Plancherel's Theorem, \eqref{eq:2.17} reduces to \eqref{eq:2.10}.

  \vspace{0.3cm}  

\noindent{\bf Sub-case 2.2 (Parallel interactions):} In this case, we assume
\begin{itemize}
    \item[(i)] $L_{\max}\leq 2^{-100}(N_{\max})^3$,
    \item[(ii)]$\left\{ \begin{array}{lll}
\text{If } N_{\min}=N_0,  \, |\sin\angle{((\xi_1,\eta_1),(\xi_2,\eta_2))}|\geq 2^{-20}, \\
\text{If } N_{\min}=N_1,  \, |\sin\angle{((\xi,\eta),(\xi_2,\eta_2))}|\geq 2^{-20},\\
\text{If } N_{\min}=N_2,  \, |\sin\angle{((\xi,\eta),(\xi_1,\eta_1))}|\geq 2^{-20}.
\end{array}\right.$
\end{itemize}

Taking these assumptions in consideration, we divide the proof of \eqref{eq:3.2} in two different sub-cases $N_{\min}=N_2$ and $N_{\min}=N_0$ since, by symmetry, the argument used for $N_{\min}=N_2$ can be applied to the case $N_{\min}=N_1$. 
\vspace{.5cm}

\noindent{\bf Sub-case 2.2.1:} \fbox{$N_{\min}=N_2$.}  With this consideration, we must have $N_0\sim N_1\sim N_{\max}$ and we will show \eqref{eq:3.3``}.

Let $A, B, C \subset \R^2$ be the following sets
\begin{equation*}
\begin{split}
    A&=\left\lbrace(|(\xi,\eta)|\cos\theta,|(\xi,\eta)|\sin\theta) \in \R^2: \min(|\theta|,|\theta-\pi|)\leq 2^{-10}\pi\right\rbrace,\\
     B&=\left\lbrace(|(\xi,\eta)|\cos\theta,|(\xi,\eta)|\sin\theta) \in \R^2: \min\left(\left|\theta-\frac{\pi}{2}\right|,\left|\theta+\frac{\pi}{2}\right|\right)\leq 2^{-10}\pi\right\rbrace,\\
      A&=\left\lbrace(|(\xi,\eta)|\cos\theta,|(\xi,\eta)|\sin\theta) \in \R^2: \min\left(\left|\theta-\frac{3\pi}{4}\right|,\left|\theta+\frac{\pi}{4}\right|\right)\leq 2^{-10}\pi\right\rbrace,\\
    \end{split}
\end{equation*}
and define $I_1, I_2, I_3 \subset \R^2 \times \R^2$ as
\begin{equation*}
    \begin{split}
        I_1&= (A\times A) \cup (B\times B)\\
        I_2&=C\times C,\\
        I_3&=(\R^2 \times \R^2)\setminus (I_1 \cup I_2).
    \end{split}
\end{equation*}

In this case, the estimate \eqref{eq:3.3``} is obtained dividing the integral in three parts involving the sets $I_1, I_2$ and $I_3$.\\

\noindent{\bf Sub-case 2.2.1.1:} \fbox{$(\xi_1,\eta_1)\times (\xi,\eta) \in I_3$.}  In this case, Kinoshita \cite{K21} (see Proposition $3.1$ there), shows that
    \begin{equation}\label{eq: estimate case 3.1.1}
    \begin{split}
        &\left|\int_{**} \widehat{u}_{N_2,L_2}(\xi_2,\eta_2,\tau_2)\chi_{I_3}((\xi_1,\eta_1),(\xi,\eta))\widehat{u}_{N_1,L_1}(\xi_1,\eta_1,\tau_1)\widehat{u}_{N_0,L_0}(\xi,\eta,\tau)d\sigma_1d\sigma\right| \\
        & \qquad \qquad \qquad \qquad \leq CN_1^{-\frac{5}{4}}L_0^{\frac{1}{4}}(L_1L_2)^{\frac{1}{2}}\|\widehat{u}_{N_0,L_0}\|_{L_{\tau,\xi,\eta}^{2}}\|\widehat{u}_{N_1,L_1}\|_{L_{\tau,\xi,\eta}^{2}}\|\widehat{u}_{N_2,L_2}\|_{L_{\tau,\xi,\eta}^{2}},
        \end{split}
    \end{equation}
where $d\sigma=d\tau d\xi d\eta$, $d\sigma_1=d\tau_1d\xi_1d\eta_1$ and $\int_{**}$ denotes the integral over the set\break $(\xi_2,\eta_2,\tau_2)=(\xi_1+\xi,\eta_1+\eta, \tau_1+\tau)$. 

For every $0<\epsilon\leq \frac{1}{8}$, it follows from \eqref{eq: estimate case 3.1.1} that
    \begin{equation}\label{eq:2.19}
    \begin{split}
        &\left|\int_{**} \widehat{u}_{N_2,L_2}(\xi_2,\eta_2,\tau_2)\chi_{I_3}((\xi_1,\eta_1),(\xi,\eta))\widehat{u}_{N_1,L_1}(\xi_1,\eta_1,\tau_1)\widehat{u}_{N_0,L_0}(\xi,\eta,\tau)d\sigma_1d\sigma\right| \\
        & \qquad \qquad \qquad \qquad \leq CN_1^{-\frac{5}{4}}L_0^{\frac{1}{2}-2\epsilon}(L_1L_2)^{\frac{1}{2}+\epsilon}\|\widehat{u}_{N_0,L_0}\|_{L_{\tau,\xi,\eta}^{2}}\|\widehat{u}_{N_1,L_1}\|_{L_{\tau,\xi,\eta}^{2}}\|\widehat{u}_{N_2,L_2}\|_{L_{\tau,\xi,\eta}^{2}}.
        \end{split}
    \end{equation}

    In view of \eqref{eq:2.19}, to get \eqref{eq:3.3``}, we just need to guarantee that
    \begin{equation}\label{eq:case 3.1}
        \frac{N_0}{N_1^{\frac{5}{4}}} \leq CN_0^{-s}N_1^s N_2^s. 
    \end{equation}
    Since $N_1\sim N_{max}$, the inequality \eqref{eq:case 3.1} follows from  the estimate \eqref{eq:2.11} obtained in {\bf Case 1}. So,  one concludes from \eqref{eq:2.19} and \eqref{eq:case 3.1} that, for every $0<\epsilon\leq \frac{1}{8}$,
        \begin{equation*}
    \begin{split}
        &\left|\int_{**} |\xi+\eta|\widehat{u}_{N_2,L_2}(\xi_2,\eta_2,\tau_2)\chi_{I_3}((\xi_1,\eta_1),(\xi,\eta))\widehat{u}_{N_1,L_1}(\xi_1,\eta_1,\tau_1)\widehat{u}_{N_0,L_0}(\xi,\eta,\tau)d\sigma_1d\sigma\right| \\
        &\;\; \leq C N_0\left|\int_{**} \widehat{u}_{N_2,L_2}(\xi_2,\eta_2,\tau_2)\chi_{I_3}((\xi_1,\eta_1),(\xi,\eta))\widehat{u}_{N_1,L_1}(\xi_1,\eta_1,\tau_1)\widehat{u}_{N_0,L_0}(\xi,\eta,\tau)d\sigma_1d\sigma\right| \\
        & \;\;  \leq CN_0^{-s}N_1^s N_2^sL_0^{\frac{1}{2}-2\epsilon}(L_1L_2)^{\frac{1}{2}+\epsilon}\|\widehat{u}_{N_0,L_0}\|_{L_{\tau,\xi,\eta}^{2}}\|\widehat{u}_{N_1,L_1}\|_{L_{\tau,\xi,\eta}^{2}}\|\widehat{u}_{N_2,L_2}\|_{L_{\tau,\xi,\eta}^{2}},
        \end{split}
    \end{equation*}
    thereby getting \eqref{eq:3.3``}.\\
    
\noindent{\bf Sub-case 2.2.1.2:} \fbox{ $(\xi_1,\eta_1)\times (\xi,\eta) \in I_2$.}
Under this condition, Proporition 3.15 in \cite{K21} states that
    \begin{equation}\label{eq: estimate case 3.1.2}
    \begin{split}
        &\left|\int_{**} |\xi+\eta|\widehat{u}_{N_2,L_2}(\xi_2,\eta_2,\tau_2)\chi_{I_2}((\xi_1,\eta_1),(\xi,\eta))\widehat{u}_{N_1,L_1}(\xi_1,\eta_1,\tau_1)\widehat{u}_{N_0,L_0}(\xi,\eta,\tau)d\sigma_1d\sigma\right| \\
        & \qquad \qquad \qquad \qquad \leq CN_1^{-\frac{1}{4}}(L_0L_1L_2)^{\frac{1}{2}}\|\widehat{u}_{N_0,L_0}\|_{L_{\tau,\xi,\eta}^{2}}\|\widehat{u}_{N_1,L_1}\|_{L_{\tau,\xi,\eta}^{2}}\|\widehat{u}_{N_2,L_2}\|_{L_{\tau,\xi,\eta}^{2}},
        \end{split}
    \end{equation}
where $d\sigma=d\tau d\xi d\eta$, $d\sigma_1=d\tau_1d\xi_1d\eta_1$ and $\int_{**}$ denotes the integral over the set $(\xi_2,\eta_2,\tau_2)=(\xi_1+\xi,\eta_1+\eta, \tau_1+\tau)$. In this case, we will show  the alternative expression \eqref{eq:3.3``}.

First we will show that,
\begin{equation}\label{eq: after 2.23`}
    N_1^{-\frac{1}{4}}L_0^{\frac{1}{2}} \leq CN_0^{-s}N_1^{s}N_2^{s}L_0^{\frac{1}{2}-2\epsilon(s)}.
\end{equation}

We divide the proof of \eqref{eq: after 2.23`} in two different cases, viz. $s \geq 0$ and $-\frac{1}{4}<s<0$.

For $s \geq 0$, one has
        \begin{equation*}
            L_0^{\frac{1}{12}} \leq (L_{\max})^{\frac{1}{12}} \leq C(N_{\max})^{\frac{1}{4}}.
        \end{equation*}
        Since $N_0\sim N_1 \sim N_{\max}$, we get
        \begin{equation}\label{eq-mm0}
            N_1^{-\frac{1}{4}}L_0^{\frac{1}{2}}\leq CN_1^{-\frac{1}{4}}(N_{\max})^{\frac{1}{4}}L_0^{\frac{5}{12}}\sim N_1^s N_0^{-s}L_0^{\frac{5}{12}}\leq N_1^s N_2^s N_0^{-s}L_0^{\frac{5}{12}}.
        \end{equation}
Considering $0<\epsilon\leq \frac{1}{24}$,  the estimate \eqref{eq-mm0} yields \eqref{eq: after 2.23`} as required.

Now, for $-\frac{1}{4}<s<0$ fix $0<r<\frac{1}{4}$ such that $s=-\frac{1}{4}+r$. Since $N_2 \leq N_0$, $N_1\sim N_{\max}$ and $L_0^{\frac{r}{3}}\leq C(N_{\max})^r$, we get
        \begin{equation}\label{eq-mm1}
            N_1^{-\frac{1}{4}}L_0^{\frac{1}{2}}\leq N_1^{-\frac{1}{4}}(N_{\max})^rL_0^{\frac{1}{2}-\frac{r}{3}} \sim N_1^{s}L_0^{\frac{1}{2}-\frac{r}{3}}\leq N_1^s \frac{N_0^{-s}}{N_2^{-s}}L_0^{\frac{1}{2}-\frac{r}{3}}=N_1^sN_2^{s} N_0^{-s}L_0^{\frac{1}{2}-\frac{r}{3}}.
        \end{equation}
Thus, we get the estimate \eqref{eq: after 2.23`} from \eqref{eq-mm1}  for every $0<\epsilon(s)\leq \frac{r}{6}=\frac{s}{6}+\frac{1}{24}$.
        
        Finally, for any $0<\epsilon(s)\leq \min\{\frac{1}{24},\frac{s}{6}+\frac{1}{24}\}$, from \eqref{eq: estimate case 3.1.2} and \eqref{eq: after 2.23`} one has
 \begin{equation*}
    \begin{split}
        &\left|\int_{**} |\xi+\eta|\widehat{u}_{N_2,L_2}(\xi_2,\eta_2,\tau_2)\chi_{I_2}((\xi_1,\eta_1),(\xi,\eta))\widehat{u}_{N_1,L_1}(\xi_1,\eta_1,\tau_1)\widehat{u}_{N_0,L_0}(\xi,\eta,\tau)d\sigma_1d\sigma\right| \\
        & \qquad \qquad   \leq CN_1^sN_2^{s} N_0^{-s}L_0^{\frac{1}{2}-\frac{r}{3}}(L_1L_2)^{\frac{1}{2}}\|\widehat{u}_{N_0,L_0}\|_{L_{\tau,\xi,\eta}^{2}}\|\widehat{u}_{N_1,L_1}\|_{L_{\tau,\xi,\eta}^{2}}\|\widehat{u}_{N_2,L_2}\|_{L_{\tau,\xi,\eta}^{2}}\\
        & \qquad \qquad   \leq CN_1^{s}N_2^{s}N_0^{-s}L_0^{\frac{1}{2}-2\epsilon(s)}(L_1L_2)^{\frac{1}{2}+\epsilon(s)}\|\widehat{u}_{N_0,L_0}\|_{L_{\tau,\xi,\eta}^{2}}\|\widehat{u}_{N_1,L_1}\|_{L_{\tau,\xi,\eta}^{2}}\|\widehat{u}_{N_2,L_2}\|_{L_{\tau,\xi,\eta}^{2}}
        \end{split}
    \end{equation*}
which is \eqref{eq:3.3``}.\\

\noindent{\bf Sub-case 2.2.1.3:} \fbox{  $(\xi_1,\eta_1)\times (\xi,\eta) \in I_1$.} 
 In this case, Proposition 3.18 in \cite{K21} implies that
     \begin{equation}\label{eq: estimate case 3.1.3}
    \begin{split}
        &\left|\int_{**} \widehat{u}_{N_2,L_2}(\xi_2,\eta_2,\tau_2)\chi_{I_1}((\xi_1,\eta_1),(\xi,\eta))\widehat{u}_{N_1,L_1}(\xi_1,\eta_1,\tau_1)\widehat{u}_{N_0,L_0}(\xi,\eta,\tau)d\sigma_1d\sigma\right| \\
        & \qquad \qquad \qquad \qquad \leq CN_1^{-1}N_2^{-\frac{1}{4}}L_0^{\frac{1}{4}}(L_1L_2)^{\frac{1}{2}}\|\widehat{u}_{N_0,L_0}\|_{L_{\tau,\xi,\eta}^{2}}\|\widehat{u}_{N_1,L_1}\|_{L_{\tau,\xi,\eta}^{2}}\|\widehat{u}_{N_2,L_2}\|_{L_{\tau,\xi,\eta}^{2}},
        \end{split}
    \end{equation}
where $d\sigma=d\tau d\xi d\eta$, $d\sigma_1=d\tau_1d\xi_1d\eta_1$ and $\int_{**}$ denotes the integral over the set\break $(\xi_2,\eta_2,\tau_2)=(\xi_1+\xi,\eta_1+\eta, \tau_1+\tau)$.
    Since $N_0 \sim N_1$, 
    \begin{equation}\label{eq:estimate case 3.1.3.1}
        N_0N_1^{-1}N_2^{-\frac{1}{4}} \sim N_2^{-\frac{1}{4}} \leq N_2^s \sim N_1^sN_2^{s} N_0^{-s}.
    \end{equation}
Considering any $0<\epsilon\leq \frac{1}{8}$, one can obtain \eqref{eq:3.3``} from \eqref{eq: estimate case 3.1.3} and \eqref{eq:estimate case 3.1.3.1} since
     \begin{equation*}
    \begin{split}
    &\left|\int_{**} |\xi+\eta|\widehat{u}_{N_2,L_2}(\xi_2,\eta_2,\tau_2)\chi_{I_1}((\xi_1,\eta_1),(\xi,\eta))\widehat{u}_{N_1,L_1}(\xi_1,\eta_1,\tau_1)\widehat{u}_{N_0,L_0}(\xi,\eta,\tau)d\sigma_1d\sigma\right| \\
        &\qquad \leq N_0\left|\int_{**} \widehat{u}_{N_2,L_2}(\xi_2,\eta_2,\tau_2)\chi_{I_1}((\xi_1,\eta_1),(\xi,\eta))\widehat{u}_{N_1,L_1}(\xi_1,\eta_1,\tau_1)\widehat{u}_{N_0,L_0}(\xi,\eta,\tau)d\sigma_1d\sigma\right| \\
        & \qquad \leq CN_0^{-s}N_1^s N_2^sL_0^{\frac{1}{2}-2\epsilon}(L_1L_2)^{\frac{1}{2}+\epsilon}\|\widehat{u}_{N_0,L_0}\|_{L_{\tau,\xi,\eta}^{2}}\|\widehat{u}_{N_1,L_1}\|_{L_{\tau,\xi,\eta}^{2}}\|\widehat{u}_{N_2,L_2}\|_{L_{\tau,\xi,\eta}^{2}}.
        \end{split}
    \end{equation*}

\vspace{0.3cm}
\noindent{\bf Sub-case 2.2.2:} \fbox{  $N_{\min}=N_0$.}
In this case, we have $N_1\sim N_2\sim N_{max}$. As in the \break {\bf Sub-case 2.2.1}, we  will prove the estimate \eqref{eq:3.2``} considering three different sub-cases.\\

\noindent{\bf Sub-case 2.2.2.1:}  \fbox{$(\xi_1,\eta_1)\times (\xi_2,\eta_2) \in I_3$.}
In this case, similarly to \eqref{eq: estimate case 3.1.1}, one has
    \begin{equation}\label{eq: estimate case 3.2.1}
    \begin{split}
        &\left|\int_{*} \widehat{u}_{N_0,L_0}(\xi,\eta,\tau)\chi_{I_3}((\xi_1,\eta_1),(\xi_2,\eta_2))\widehat{u}_{N_1,L_1}(\xi_1,\eta_1,\tau_1)\widehat{u}_{N_2,L_2}(\xi_2,\eta_2,\tau_2)d\sigma_1d\sigma_2\right| \\
        & \qquad \qquad \qquad \qquad \leq CN_1^{-\frac{5}{4}}L_2^{\frac{1}{4}}(L_1L_0)^{\frac{1}{2}}\|\widehat{u}_{N_0,L_0}\|_{L_{\tau,\xi,\eta}^{2}}\|\widehat{u}_{N_1,L_1}\|_{L_{\tau,\xi,\eta}^{2}}\|\widehat{u}_{N_2,L_2}\|_{L_{\tau,\xi,\eta}^{2}},
        \end{split}
    \end{equation}
where $d\sigma_j=d\tau_j d\xi_j d\eta_j$ and $\int_*$ denotes the integral over the set $(\xi,\eta,\tau)=(\xi_1+\xi_2,\eta_1+\eta_2,\tau_1+\tau_2)$.

First, we will show that
\begin{equation}\label{eq: after 2.23}
    N_0N_1^{-\frac{5}{4}}L_0^{\frac{1}{2}} \leq CN_0^{-s}N_1^{s}N_2^{s}L_0^{\frac{1}{2}-2\epsilon(s)}.
\end{equation}

As in the proof of  \eqref{eq: after 2.23`} we divide the proof of \eqref{eq: after 2.23} in two different  cases.

For $s \geq 0$, let $s=-\frac{1}{4}+r$ for a fixed  $r \geq \frac{1}{4}$. Then
    \begin{equation}\label{eq:s>0}
    \begin{split}
        N_0N_1^{-\frac{5}{4}}L_0^{\frac{1}{2}}        &\leq N_0^{\frac{1}{4}}N_1^{-\frac{1}{2}}L_0^{\frac{1}{2}}\\
        &\leq CN_0^{\frac{1}{4}}N_1^{-\frac{1}{2}}N_{\max}^{\frac{1}{4}}L_0^{\frac{5}{12}}\\
        &\leq CN_0^{\frac{1}{4}}N_1^{-\frac{1}{2}}N_{2}^{r}L_0^{\frac{5}{12}}\\
        &\leq CN_0^{\frac{1}{4}-r}N_1^{r}N_1^{-\frac{1}{4}}N_2^{-\frac{1}{4}}N_{2}^{r}L_0^{\frac{5}{12}}\\
        &=CN_0^{-s}N_1^{s}N_2^{s}L_0^{\frac{5}{12}}.
        \end{split}
    \end{equation}
    Thus,  for every $0<\epsilon(s)\leq \frac{1}{24}$ we get the estimate \eqref{eq: after 2.23} from \eqref{eq:s>0}.
    
Now, for $-\frac{1}{4}<s<0$, let $s=-\frac{1}{4}+r$ for a fixed $0<r< \frac{1}{4}$. Then
    \begin{equation}\label{eq:s<0}
        \begin{split}
        N_0N_1^{-\frac{5}{4}}L_0^{\frac{1}{2}}        &\leq N_0^{\frac{1}{4}}N_1^{-\frac{1}{2}}L_0^{\frac{1}{2}}\\
        &\leq N_0^{\frac{1}{4}}N_1^{-\frac{1}{2}}N_{\max}^{r}L_0^{\frac{1}{2}-\frac{r}{3}}\\
        &\sim N_0^{\frac{1}{4}}N_1^{-\frac{1}{4}}N_2^{-\frac{1}{4}}N_{2}^{r}L_0^{\frac{1}{2}-\frac{r}{3}}\\
        &\leq N_0^{\frac{1}{4}-r}N_1^{r}N_1^{-\frac{1}{4}}N_2^{-\frac{1}{4}}N_{2}^{r}L_0^{\frac{1}{2}-\frac{r}{3}}\\
        &=N_0^{-s}N_1^{s}N_2^{s}L_0^{\frac{1}{2}-\frac{r}{3}}.
        \end{split}
    \end{equation}
     Considering $0<\epsilon(s)\leq\frac{r}{6}=\frac{s}{6}+\frac{1}{24}$, we obtain \eqref{eq: after 2.23} from \eqref{eq:s<0}.

Finally, combining \eqref{eq: estimate case 3.2.1} and \eqref{eq: after 2.23}, we get
    \begin{equation*}
    \begin{split}
    &\left|\int_{*} |\xi+\eta|\widehat{u}_{N_0,L_0}(\xi,\eta,\tau)\chi_{I_3}((\xi_1,\eta_1),(\xi_2,\eta_2))\widehat{u}_{N_1,L_1}(\xi_1,\eta_1,\tau_1)\widehat{u}_{N_2,L_2}(\xi_2,\eta_2,\tau_2)d\sigma_1d\sigma_2\right| \\
        &\qquad\leq CN_0\left|\int_{*} \widehat{u}_{N_0,L_0}(\xi,\eta,\tau)\chi_{I_3}((\xi_1,\eta_1),(\xi_2,\eta_2))\widehat{u}_{N_1,L_1}(\xi_1,\eta_1,\tau_1)\widehat{u}_{N_2,L_2}(\xi_2,\eta_2,\tau_2)d\sigma_1d\sigma_2\right| \\
        & \qquad \leq CN_0N_1^{-\frac{5}{4}}L_2^{\frac{1}{4}}(L_1L_0)^{\frac{1}{2}}\|\widehat{u}_{N_0,L_0}\|_{L_{\tau,\xi,\eta}^{2}}\|\widehat{u}_{N_1,L_1}\|_{L_{\tau,\xi,\eta}^{2}}\|\widehat{u}_{N_2,L_2}\|_{L_{\tau,\xi,\eta}^{2}}\\
        & \qquad \leq CN_0^{-s}N_1^{s}N_2^{s}L_0^{\frac{1}{2}-2\epsilon(s)}(L_1L_2)^{\frac{1}{2}}\|\widehat{u}_{N_0,L_0}\|_{L_{\tau,\xi,\eta}^{2}}\|\widehat{u}_{N_1,L_1}\|_{L_{\tau,\xi,\eta}^{2}}\|\widehat{u}_{N_2,L_2}\|_{L_{\tau,\xi,\eta}^{2}}\\
        \end{split}
    \end{equation*}
for every $0<\epsilon(s)\leq \min\{\frac{1}{24},\frac{s}{6}+\frac{1}{24}\}$, thereby obtaining \eqref{eq:3.2``}.\\

\noindent{\bf Sub-case 2.2.2.2:}  \fbox{ $(\xi_1,\eta_1)\times (\xi_2,\eta_2) \in I_2$.} 
In this case, Kinoshita in \cite{K21} (see Proposition $3.25$ there), shows   that
    \begin{equation}\label{eq: estimate case 3.2.2}
    \begin{split}
        &\left|\int_{*} |\xi+\eta|\widehat{u}_{N_0,L_0}(\xi,\eta,\tau)\chi_{I_2}((\xi_1,\eta_1),(\xi_2,\eta_2))\widehat{u}_{N_1,L_1}(\xi_1,\eta_1,\tau_1)\widehat{u}_{N_2,L_2}(\xi_2,\eta_2,\tau_2)d\sigma_1d\sigma_2\right| \\
        & \qquad \qquad \qquad \qquad \leq CN_0^{\frac{1}{4}}N_1^{-\frac{1}{2}}(L_0L_1L_2)^{\frac{1}{2}}\|\widehat{u}_{N_0,L_0}\|_{L_{\tau,\xi,\eta}^{2}}\|\widehat{u}_{N_1,L_1}\|_{L_{\tau,\xi,\eta}^{2}}\|\widehat{u}_{N_2,L_2}\|_{L_{\tau,\xi,\eta}^{2}},
        \end{split}
    \end{equation}
where $d\sigma_j=d\tau_j d\xi_j d\eta_j$ and $\int_*$ denotes the integral over the set $(\xi,\eta,\tau)=(\xi_1+\xi_2,\eta_1+\eta_2,\tau_1+\tau_2)$. Now, it follows from the second line of equations \eqref{eq:s>0} and \eqref{eq:s<0} that
\begin{equation*}
    N_0^{\frac{1}{4}}N_1^{-\frac{1}{2}}L_0^{\frac{1}{2}}\leq CN_0^{-s}N_1^{s}N_2^{s}L_0^{\frac{1}{2}-2\epsilon(s)}
\end{equation*}
for every $0<\epsilon(s) \leq \frac{1}{24}$, when $s\geq 0$, and for every $0<\epsilon(s) \leq \frac{s}{6}+\frac{1}{24}$, when $-\frac{1}{4}<s< 0$. Consequently, we obtain \eqref{eq:3.2``} by the same argument used in the previous case.\\

\noindent{\bf Sub-case 2.2.2.3:}  \fbox{ $(\xi_1,\eta_1)\times (\xi_2,\eta_2) \in I_1$.}
Similarly to \eqref{eq: estimate case 3.1.3}, one has
     \begin{equation}\label{eq: estimate case 3.2.3}
    \begin{split}
        &\left|\int_{*} \widehat{u}_{N_0,L_0}(\xi,\eta,\tau)\chi_{I_1}((\xi_1,\eta_1),(\xi_2,\eta_2))\widehat{u}_{N_1,L_1}(\xi_1,\eta_1,\tau_1)\widehat{u}_{N_2,L_2}(\xi_2,\eta_2,\tau_2)d\sigma_1d\sigma_2\right| \\
        & \qquad \qquad \qquad \qquad \leq CN_1^{-1}N_0^{-\frac{1}{4}}L_2^{\frac{1}{4}}(L_1L_0)^{\frac{1}{2}}\|\widehat{u}_{N_0,L_0}\|_{L_{\tau,\xi,\eta}^{2}}\|\widehat{u}_{N_1,L_1}\|_{L_{\tau,\xi,\eta}^{2}}\|\widehat{u}_{N_2,L_2}\|_{L_{\tau,\xi,\eta}^{2}},
        \end{split}
    \end{equation}
where $d\sigma_j=d\tau_j d\xi_j d\eta_j$ and $\int_*$ denotes the integral over the set $(\xi,\eta,\tau)=(\xi_1+\xi_2,\eta_1+\eta_2,\tau_1+\tau_2)$.
    Note that
\begin{equation*}
    N_0N_1^{-1}N_0^{-\frac{1}{4}}L_0^{\frac{1}{2}}=N_0^{\frac{3}{4}}N_1^{-1}L_0^{\frac{1}{2}}\leq N_0^{\frac{1}{4}}N_1^{\frac{1}{2}}N_1^{-1}L_0^{\frac{1}{2}}=N_0^{\frac{1}{4}}N_1^{-\frac{1}{2}}L_0^{\frac{1}{2}}
\end{equation*}
and the estimate \eqref{eq:3.2``} follows from the same argument used in the preceding cases.\\

Finally, we analyse the values that the choice of $\epsilon(s)$ can assume in order to satisfy \eqref{eq:estimate partial derivative ZK}. Taking in consideration the various cases outlined previously, it becomes clear that for $s\geq 0$, one can take any $\epsilon=\epsilon(s)$ within the interval $\left(0,\frac{1}{24}\right]$. On the other hand, for $-\frac{1}{4}<s<0$, it has been demonstrated that in some cases, $\epsilon(s)$ can range across the interval $\left(0,\frac{1}{24}\right]$, while in others, it must adhere to $\epsilon(s) \in \left(0,\frac{s}{6}+\frac{1}{24}\right]$. In conclusion, given $s>-\frac{1}{4}$, the estimate \eqref{eq:estimate partial derivative ZK} is valid for any $0<\epsilon(s)\leq \min\left\lbrace\frac{1}{24},\frac{s}{6}+\frac{1}{24}\right\rbrace$. From this consideration, it is obvious that $\epsilon(s)\leq \epsilon(0)$ as advertised.
\end{proof}

The following result is the analytic version of Lemma \ref{lemma:estimates partial ZK} in Gevrey- Bourgain's space.
\begin{proposition}\label{prop: estimates ZK} Let  $\sigma \geq 0$ and $s>-\frac{1}{4}$. Then, for $\epsilon (s)=min(\frac{1}{24}, \frac{s}{6}+\frac{1}{24}) $, we have
\begin{equation}\label{eq:estimate partial derivative sigma}
\|(\partial_x+\partial_y)(u_1u_2)\|_{X^{\sigma,s,-\frac{1}{2}+2\epsilon(s)}} \leq C\|u_1\|_{X^{\sigma,s,\frac{1}{2}+\epsilon(s)}}\|u_2\|_{X^{\sigma,s,\frac{1}{2}+\epsilon(s)}},
\end{equation}
where $C>0$ depends on $s$.
\end{proposition}
\begin{proof}
The proof of this proposition follows by applying the inequality $e^{\sigma|\gamma|}\leq e^{\sigma|\gamma-\gamma_{1}-\gamma_2|}e^{\sigma|\gamma_1|}e^{\sigma|\gamma_2|}$ and the estimate \eqref{eq:estimate partial derivative ZK} for $e^{\sigma|D_{x,y}|}u_i$ in place of $u_i$ for $i=1,2$. For a more detailed proof we refer to \cite{BFH}.
\end{proof}

Concerning the IVP \eqref{eq:gZK} with $k=2$, Kinoshita in \cite{K22} proved the following trilinear estimate in order to obtain the local well-posedness in $H^s(\R^2)$ for $s\geq \frac{1}{4}$.
\begin{lemma} (\cite{K22})  \label{lemma:estimates partial mZK} Let $s\geq\frac{1}{4}$. Then, for all $0 <\epsilon \ll 1$, we have
\begin{equation}\label{eq:estimate partial derivative mZK}
\|(\partial_x+\partial_y)(u_1 u_2 u_3)\|_{X^{s,-\frac{1}{2}+2\epsilon}} \leq C\prod_{i=1}^{3}\|u_{i}\|_{X^{s,\frac{1}{2}+\epsilon}},
\end{equation}
where $C>0$ depends on $s$ and $\epsilon$.
\end{lemma}

The following is the analytic version of the previous result.

\begin{proposition}\label{prop:estimates} Let  $\sigma \geq 0$ and $s\geq\frac{1}{4}$ be given.Then, for all $0 <\epsilon \ll 1$, we have
\begin{equation}\label{eq:partial estimates mZK}
\|(\partial_x+\partial_y)(u_1 u_2 u_3)\|_{X^{\sigma,   s,  -\frac{1}{2}+2\epsilon}} \leq C\prod_{i=1}^{3}\|u_{i}\|_{X^{\sigma,  s,\frac{1}{2}+\epsilon}},
\end{equation}
where $C>0$ depends on $s$ and $\epsilon$.
\end{proposition}
\begin{proof}
The proof follows the same ideia of the proof of Propostion \ref{prop: estimates ZK}.
\end{proof}

\section{Local well-posedness - Proof of Theorems \ref{teo:local ZK} and \ref{teo:local mZK}}\label{sec:3}
In this section we use the estimates from the previous section and provide  proofs for the local well-posedness results to the IVP \eqref{eq:gZK} with $k=1$ and $k=2$ and for given real analytic initial data. For the sake of completeness we provide a proof for the case $k=1$. The proof for the case $k=2$ follows similarly.

\begin{proof}[Proof of Theorem \ref{teo:local ZK}]
Let $\sigma>0$, $k=1$ and $u_0 \in G^{\sigma,   s}(\R^2)$ with $s>-\frac{1}{4}$. For $0<T\leq 1$, let $\psi_T$ be the cut-off function defined in \eqref{eq:cut-off} and consider a solution map given by
\begin{equation*}
\Phi_T(u)=\psi(t)W(t)u_0-\psi_T(t)\int_0^t W(t-t')\mu a (\partial_x+\partial_y)u^{2}dt'.
\end{equation*}
Our main goal is to show that there are $b>\frac{1}{2}$, $r>0$ and $T_0=T_0(\|u_0\|_{G^{\sigma, s}})>0$ such that $\Phi_{T_0}:B(r) \rightarrow B(r)$ is a contraction map, where
\begin{equation*}
B(r)=\left\lbrace u \in X_{T_0}^{\sigma,   s,   b}: \|u\|_{X^{\sigma,  s,   b}}\leq r\right\rbrace.
\end{equation*}

For $u \in B(r)$, applying the linear inequalities \eqref{eq:semigroup 1} and \eqref{eq:semigroup 2} and the bilinear estimate \eqref{eq:estimate partial derivative sigma} with $b=\frac{1}{2}+\epsilon(s)$ and $b'=\frac{1}{2}+2\epsilon(s)$ as in the Proposition \ref{prop: estimates ZK}, we obtain
\begin{equation}\label{eq: well defined}
\begin{split}
\|\Phi_{T_0}(u)\|_{X^{\sigma,  s,   b}} & \leq \|\psi(t)W(t)u_0\|_{X^{\sigma,  s,   b}}+\left\|\psi_{T_0}(t)\int_0^t W(t-t')\mu a(\partial_x+\partial_y)u^{2}dt'\right\|_{X^{\sigma,  s,   b}}\\
&\leq C\|u_0\|_{G^{\sigma,   s}}+CT_0^{\frac{1}{d}}\|(\partial_x+\partial_y)u^{2}\|_{X^{\sigma,  s,   b'-1}} \\
&\leq C\|u_0\|_{G^{\sigma,   s}}+CT_0^{\frac{1}{d}}\|u\|_{X^{\sigma,  s,   b}}^{2} \\
&\leq \frac{r}{2}+CT_0^{\frac{1}{d}}r^{2},
\end{split}
\end{equation}
where $\frac{1}{d}=b'-b$ and $r=2C\|u_0\|_{G^{\sigma,   s}}$.

By choosing
\begin{equation}\label{eq: condition T0 1}
T_0 \leq \frac{1}{(2Cr)^d},
\end{equation}
one gets from \eqref{eq: well defined} that $\|\Phi_{T_0}(u)\|_{X^{\sigma,  s,   b}} \leq r$ for all $u \in B(r)$ showing that $\Phi_{T_0}(B(r))\subset B(r)$.

Now, for $u$, $v \in B(r)$, using \eqref{eq:semigroup 2} once again, we have
\begin{equation*}
\|\Phi_{T_0}(u)-\Phi_{T_0}(v)\|_{X^{\sigma,  s,   b}} \leq CT_{0}^{\frac{1}{d}}\|(\partial_x+\partial_y)u^2-(\partial_x+\partial_y)v^{2}\|_{X^{\sigma,  s,   b'-1}}.
\end{equation*}
Since
\begin{equation*}
u^{2}-v^{2}=(u-v)(u+v),
\end{equation*}
from  \eqref{eq:estimate partial derivative sigma}, we conclude that
\begin{equation*}
\begin{split}
\|\Phi_{T_0}(u)-\Phi_{T_0}(v)\|_{X^{\sigma,  s,   b}} &\leq CT_{0}^{\frac{1}{d}}\left\|(\partial_x+\partial_y)\left((u-v)(u+v)\right)\right\|_{X^{\sigma,  s,   b'-1}}\\
&\leq CT_{0}^{\frac{1}{d}}\|u-v\|_{X^{\sigma,  s,   b}}\left(\|u\|_{X^{\sigma,  s,   b}}+\|v\|_{X^{\sigma,  s,   b}}\right)\\
&\leq CT_{0}^{\frac{1}{d}}r\|u-v\|_{X^{\sigma,  s,   b}}
\end{split}
\end{equation*}
By choosing 
\begin{equation}\label{eq: condition T0 2}
T_0 \leq \frac{1}{(2Cr)^d},
\end{equation}
it follows that
\begin{equation*}
\|\Phi_{T_0}(u)-\Phi_{T_0}(v)\|_{X^{\sigma,  s,   b}} \leq \frac{1}{2}\|u-v\|_{X^{\sigma,  s,   b}}
\end{equation*}
and $\Phi_{T_0}$ is a contraction map.

Finally, it is sufficient to choose $0<T_0\leq 1$ satisfying \eqref{eq: condition T0 1} and \eqref{eq: condition T0 2}. More precisely, considering
\begin{equation}\label{eq:T0}
T_0=\frac{c_0}{(1+\|u_0\|_{G^{\sigma,   s}}^2)^{\frac{d}{2}}}
\end{equation}  for an appropriate constant $c_0>0$ depending on $s$ and $b$, we conclude that $\Phi_{T_0}$ admits a unique fixed point which is a local solution of the IVP \eqref{eq:new gZK}. Moreover, the solution satisfies 
\begin{equation}\label{eq:bound for solution ZK}
\|u\|_{X_{T_0}^{\sigma,  s,   \frac{1}{2}+\epsilon(s)}}\leq r=C\|u_0\|_{G^{\sigma,   s}},
\end{equation}
where $\epsilon (s)=\min(\frac{1}{24}, \frac{s}{6}+\frac{1}{24})$.

The continuous dependence on the initial data follows by a similar argument and the proof is complete.
\end{proof}

\begin{proof}[Proof of Theorem \ref{teo:local mZK}]
    The proof is similar to that of Theorem \ref{teo:local ZK}. The only difference it that in this case we use the trilinear estimate \eqref{eq:partial estimates mZK} from Proposition \ref{prop:estimates}. So, we omit the details.

\end{proof}

\section{Almost conserved quantity}\label{sec:4}
This section is dedicated to define almost conserved quantities associated to the ZK equation and the mZK equation and find their growth estimates  in order to apply the local results repeatedly in adequate small time subintervals to cover arbitrary time interval $[-T,   T]$, for any $T>0$. Taking in consideration the conserved quantities \eqref{eq:mass conversation} and \eqref{eq:new energy conservation}, for the ZK equation, we define
\begin{equation}\label{eq:almost conserved ZK}
M_\sigma (t)= \|u(t)\|^{2}_{G^{\sigma,0}}.
\end{equation}
and for the mZK equation, we define
\begin{equation}\label{eq:almost conserved mZK}
E_\sigma (t)= \|u(t)\|^{2}_{G^{\sigma,1}}-\int \partial_x (e^{\sigma |D_{x,y}|}u)\partial_y (e^{\sigma |D_{x,y}|}u)dxdy-  \frac{\mu a}{2}\|e^{\sigma |D_{x,y}|}u\|_{L_{x,y}^4}^4.
\end{equation}

Note that, for $\sigma=0$, \eqref{eq:almost conserved ZK} and \eqref{eq:almost conserved mZK} turn out to be the conserved quantities \eqref{eq:mass conversation} and \eqref{eq:new energy conservation}, respectively. However, for $\sigma>0$ these quantities fail to be conserved in time and we will prove that  they are almost conserved by establishing growth estimates. 

\subsection{Almost conserved quantity at $L^2$-level of Sobolev regularity} 

In this subsection, we find a growth estimate for $M_\sigma (t)$ defined in \eqref{eq:almost conserved ZK} and consequently prove that it is an almost conserved quantity at the $L^2$-level of Sobolev regularity. 

Denoting $U=e^{\sigma|D_{x,y}|}u$ and differentiating $M_\sigma (t)$ given by \eqref{eq:almost conserved ZK} with respect to $t$, we obtain 
\begin{equation}\label{eq:partial M}
\begin{split}
        \frac{d}{dt}(M_\sigma(t))&=2\int U \partial_tUdxdy.
\end{split}
    \end{equation}
    Applying the operator $e^{\sigma|D_{x,y}|}$ to the ZK equation \eqref{eq:new gZK} with $k=1$, we get
    \begin{equation}\label{eq:op in ZK}
        \partial_t U +(\partial_x^3+\partial_y^3)U+\mu a(\partial_x+\partial_y) U^2=F(U),
        \end{equation}
        where $F(U)$ is given by
\begin{equation}\label{eq:F}
F(U)=\mu a (\partial_x+\partial_y)\left[U^{2}-e^{\sigma |D_{x,y}|}\big((e^{-\sigma |D_{x,y}|}U)^{2}\big)\right].
\end{equation}       
 Using \eqref{eq:op in ZK} in \eqref{eq:partial M}, one has
\begin{equation*}
    \begin{split}
            \frac{d}{dt} M_{\sigma}(t)=&-2\int U\partial_x^3Udxdy-2\int U\partial_y^3Udxdy-2\mu a \int U \partial_xU^{2}dxdy-\\
            &-2\mu a\int U \partial_yU^{2}dxdy+2\int UF(U)dxdy.
    \end{split}
\end{equation*}
We can assume that $U$ and all its partial derivatives tend to zero as $|(x,y)| \rightarrow \infty$ (see \cite{SS} for a detailed argument). With this consideration, it follows from integration by parts that
\begin{equation} \label{eq:partial M new}
       \frac{d}{dt}  M_{\sigma}(t)=2\int UF(U)dxdy.
\end{equation}

Integrating \eqref{eq:partial M new} in time over $[0, t']$ for $0<t'\leq T$, we obtain
\begin{equation}\label{eq:estimate M}
    M_\sigma(t') =M_\sigma(0)+R_\sigma(t'),
\end{equation}
where
\begin{equation}\label{eq:R ZK}
\begin{split}
    R_\sigma(t')=&  \, 2\int \int \chi_{[0, t']}UF(U)dxdydt.
    \end{split} 
\end{equation}

Our objective is to find appropriate estimates for $R_{\sigma}(t')$ and use it to get control on the growth of $M_{\sigma}(t')$. For this purpose, we first find estimates for $F(U)$ in the Bourgain's space norm. From Lemma \ref{lemma:estimates partial ZK}, for $s>-\frac{1}{4}$ and $\epsilon(s)=min(\frac{1}{24}, \frac{s}{6}+\frac{1}{24}) $, one has
\begin{equation}\label{eq:remark}
    \|(\partial_x+\partial_y)(uv)\|_{X^{s,-\frac{1}{2}+\epsilon(s)}} \leq C\|u\|_{X^{s,\frac{1}{2}+\epsilon(s)}}\|v\|_{X^{s,\frac{1}{2}+\epsilon(s)}}.
\end{equation}

If we consider $\widehat{f}(\gamma,\tau)=\frac{|\widehat{u}|(\gamma,\tau)}{\langle \gamma\rangle^{s}}$ and $\widehat{g}(\gamma,\tau)=\frac{|\widehat{v}|(\gamma,\tau)}{\langle \gamma\rangle^{s}}$, equation \eqref{eq:remark} implies
\begin{equation}\label{eq:equiv est}
\begin{split}
    \Bigg|\Bigg| \frac{(\xi+\eta)\langle\gamma\rangle^s}{\langle \tau-\xi^3-\eta^3\rangle^{\frac{1}{2}-\epsilon(s)}}  \int_{\R^3} \widehat{f}(\gamma-\gamma_1,\tau-\tau_1)\widehat{g}(\gamma_1,\tau_1)d\gamma_1d\tau_1 \Bigg|\Bigg|_{L^2_{\tau,\xi,\eta}} &\leq C\|f\|_{X^{s,\frac{1}{2}+\epsilon(s)}}\|g\|_{X^{s,\frac{1}{2}+\epsilon(s)}}\\
    & \leq C\|u\|_{X^{0,\frac{1}{2}+\epsilon(s)}}\|v\|_{X^{0,\frac{1}{2}+\epsilon(s)}}.
\end{split}
\end{equation}

Using this estimate, one can prove the following result.

\begin{lemma}\label{lemma: estimate B}
    For $\theta \in [0, \frac{1}{4})$, consider the bilinear operator defined via Fourier transform  by
\begin{equation*}
    \widehat{B_\theta(u,v)}(\gamma,\tau)=\int_{\R^3} [\min(|\gamma-\gamma_1|,|\gamma_1|)]^{\theta}\widehat{u}(\gamma-\gamma_1,\tau-\tau_1)\widehat{v}(\gamma_1,\tau_1)d\gamma_1d\tau_1.
\end{equation*}
Then, for $\epsilon(-\theta)=\min(\frac{1}{24}, \frac{-\theta}{6}+\frac{1}{24})=\frac{-\theta}{6}+\frac{1}{24}$, one has
\begin{equation*}
    \|(\partial_x+\partial_y)B_\theta(u,v)\|_{X^{0,-\frac{1}{2}+\epsilon(-\theta)}} \leq C\|u\|_{X^{0,\frac{1}{2}+\epsilon(-\theta)}}\|v\|_{X^{0,\frac{1}{2}+\epsilon(-\theta)}}
\end{equation*}
\end{lemma}
\begin{proof}

From triangular inequality, it can be shown that
\begin{equation}\label{eq:min ineq}
    \min(|\gamma-\gamma_1|,|\gamma_1|)\leq C\frac{\langle \gamma-\gamma_1 \rangle\langle \gamma_1 \rangle}{\langle \gamma\rangle}
\end{equation}
So, using \eqref{eq:min ineq} and \eqref{eq:equiv est} with $s=-\theta$, for $\theta \in [0, \frac{1}{4})$ we have
\begin{equation*}
\begin{split}
    &\|(\partial_x+\partial_y)B_\theta(u,v)\|_{X^{s,-\frac{1}{2}+\epsilon(-\theta)}}\leq \\
    & \qquad\qquad\qquad \leq C^{\theta}\Bigg|\Bigg|\frac{(\xi+\eta)\langle \gamma\rangle^{-\theta}}{\langle\tau - \xi^3-\eta^3\rangle^{\frac{1}{2}-\epsilon(-\theta)}}\int_{\R^3}\frac{|\widehat{u}|(\gamma-\gamma_1,\tau-\tau_1)}{\langle\gamma-\gamma_1\rangle^{-\theta}}\frac{|\widehat{v}|(\gamma_1,\tau_1)}{\langle\gamma_1\rangle^{-\theta}}d\gamma_1d\tau_1\Bigg|\Bigg|_{L^2_{\tau,\xi,\eta}}\\
    &\qquad \qquad \qquad\leq C\|u\|_{X^{0,\frac{1}{2}+\epsilon(-\theta)}}\|v\|_{X^{0,\frac{1}{2}+\epsilon(-\theta)}}.
\end{split}
\end{equation*}
\end{proof}

Now we move to find estimate for $F$ defined in \eqref{eq:F}  in the Gevrey- Bourgain's space norm.
\begin{lemma}\label{lemma: estimates F}
Consider $F$ be as defined in \eqref{eq:F}, and let $\sigma >0$ and $\theta \in \left[0,\frac{1}{4}\right)$. Then, for \break$\epsilon(-\theta)=\min(\frac{1}{24}, \frac{-\theta}{6}+\frac{1}{24})=\frac{-\theta}{6}+\frac{1}{24}$, one has
\begin{equation}\label{eq: estimate F in L2}
\|F(U)\|_{X^{0,-\frac{1}{2}+\epsilon(-\theta)}} \leq C \sigma^{\theta} \|U\|_{X^{0,   \frac{1}{2}+\epsilon(-\theta)}}^{2}.
\end{equation}
\end{lemma}
\begin{proof}
First, observe that
\begin{equation}\label{eq: transf F}
\begin{split}
|\widehat{F(U)}(\gamma, \tau)| \leq C|\xi + \eta|\int_{\R^3}\big(1-e^{\sigma(|\gamma|-|\gamma-\gamma_1|-|\gamma_1|)}\big)|\widehat{U}(\gamma-\gamma_1,\tau-\tau_1)| \, |\widehat{U}(\gamma_1,\tau_1)|d\gamma_1 d\tau_1.
\end{split}
\end{equation}

Now, from Lemma \ref{lemma: inequality for e}, we get
\begin{equation}\label{eq: 1-e ZK}
    1-e^{\sigma(|\gamma|-|\gamma-\gamma_1|-|\gamma_1|)} \leq 2^{\theta}\sigma^{\theta}[min(|\gamma-\gamma_1|,|\gamma_1|)]^{\theta}.
\end{equation}
Consequently, using \eqref{eq: 1-e ZK} in \eqref{eq: transf F}, we have
\begin{equation}\label{eq: tranf F nova}
\begin{split}
|\widehat{F(U)}(\gamma, \tau)| &\leq C\sigma^{\theta}|\xi + \eta|\int_{\R^3}[min(|\gamma-\gamma_1|,|\gamma_1|)]^{\theta}|\widehat{U}(\gamma-\gamma_1,\tau-\tau_1)|  \,|\widehat{U}(\gamma_1,\tau_1)|d\gamma_1 d\tau_1\\
&=C\sigma^{\theta} |(\xi+\eta)\widehat{B_{\theta}(W,W)}|,
\end{split}
\end{equation}
where $\widehat{W}=|\widehat{U}|$.

Using Lemma \ref{lemma: estimate B}, it follows from \eqref{eq: tranf F nova} that
\begin{equation*}
\begin{split}
    \|F(U)\|_{X^{0,-\frac{1}{2}+\epsilon(-\theta)}}&\leq C\sigma^{\theta}\|(\partial_x+\partial_y)B_{\theta}(W,W)\|_{X^{0,-\frac{1}{2}+\epsilon(-\theta)}}\\
    &\leq  C\sigma^{\theta}\|W\|_{X^{0,\frac{1}{2}+\epsilon(-\theta)}}^2\\
    &\leq  C\sigma^{\theta}\|U\|_{X^{0,\frac{1}{2}+\epsilon(-\theta)}}^2,
    \end{split}
\end{equation*}
as required.
\end{proof}

In what follows,  we use the estimate obtained in Lemma \ref{lemma: estimates F} to prove that the quantity $M_{\sigma}(t)$ defined in \eqref{eq:almost conserved ZK} is almost conserved. We start with the following result.

\begin{proposition}\label{prop:conserved quantity ZK}
    Let $\sigma>0$ and $\theta \in [0,   \frac{1}{4})$. Then, for $\epsilon(-\theta)=\min(\frac{1}{24}, \frac{-\theta}{6}+\frac{1}{24})=\frac{-\theta}{6}+\frac{1}{24}$, there exists $C>0$ such that for any solution $u \in X_T^{\sigma,  0,   \frac{1}{2}+\epsilon(-\theta)}$ to the IVP \eqref{eq:new gZK} with $k=1$ in the interval $[0, T]$, we have
\begin{equation}\label{eq:conserved quantity ZK}
    \sup_{t\in [0, T]} M_\sigma(t) \leq M_\sigma (0) + C\sigma^{\theta}\|u\|_{X_T^{\sigma,  0,   \frac{1}{2}+\epsilon(-\theta)}}^{3}.
\end{equation}
\end{proposition}
\begin{proof}
 Taking in consideration the identities  \eqref{eq:estimate M} and \eqref{eq:R ZK}, to prove \eqref{eq:conserved quantity ZK} we have to estimate $|R_\sigma(t')|$ for all $0<t'\leq T$. For this purpose, we use the Cauchy-Schwarz inequality followed by Lemma \ref{lemma:restriction} and the estimate \eqref{eq: estimate F in L2} restricted to the time slab and obtain that for \break $\epsilon(-\theta)=\min(\frac{1}{24}, \frac{-\theta}{6}+\frac{1}{24})=\frac{-\theta}{6}+\frac{1}{24}$, there exists $C>0$ such that
\begin{equation}\label{eq:estimate first term F}
    \begin{split}
    \left|\int \int \chi_{[0, t']}UF(U)dxdydt\right| &\leq \|\chi_{[0, t']}U\|_{X^{  0,   \frac{1}{2}-\epsilon(-\theta)}}\|\chi_{[0, t']}F(U)\|_{X^{  0,  -\frac{1}{2}+\epsilon(-\theta)}}\\
    &\leq \|U\|_{X_{T}^{  0,   \frac{1}{2}+\epsilon(-\theta)}}\|F(U)\|_{X_{T}^{  0,   -\frac{1}{2}+\epsilon(-\theta)}}\\
    &\leq C\sigma^{\theta}\|u\|^3_{X_{T}^{ \sigma,  0,   \frac{1}{2}+\epsilon(-\theta)}}
    \end{split}
\end{equation}
for all $0<t'\leq T$. Using the estimate \eqref{eq:estimate first term F} and taking the supremum over $0<t'\leq T$ in \eqref{eq:estimate M}, it follows that
\begin{equation*}
    \begin{split}
    \sup_{t'\in [0, T]} M_\sigma (t') &\leq M_\sigma (0) + \sup_{t'\in [0, T]}|R_\sigma(t')|\\
    & \leq M_\sigma (0) + C\sigma^{\theta}\|u\|^3_{X_{T}^{ \sigma,  0,   \frac{1}{2}+\epsilon(-\theta)}},
    \end{split}
\end{equation*}
as required.
\end{proof}

As a corollary, now we prove that the quantity  $M_\sigma (t)$ is an almost conserved quantity.

\begin{corollary}\label{cor: almost conserved quantity ZK}
Let $\sigma>0$ and $\theta \in [0,\frac{1}{4})$. Then, there exists $C>0$ such that the solution $u \in X_T^{ \sigma,   0,   \frac{1}{2}+\frac{1}{24}}$ to the IVP \eqref{eq:new gZK} with $k=1$ given by Theorem \ref{teo:local ZK} satisfies
\begin{equation}\label{eq:estimate mass}
    \sup_{t\in [0, T]} M_\sigma (t) \leq M_\sigma (0) + C\sigma^{\theta} M_\sigma (0)^{\frac{3}{2}}.
\end{equation}
\end{corollary}
\begin{proof}
From Theorem \ref{teo:local ZK} with $s=0$, the local solution $u$ to the IVP \eqref{eq:new gZK} with $k=1$ in the interval $[0, T]$ belongs to $X_T^{\sigma,  0,   \frac{1}{2}+\epsilon(0)}$ where $\epsilon(0)=\frac{1}{24}$ and, from \eqref{eq:bound for solution ZK teo}, it satisfies 
\begin{equation*}
    \|u\|_{X_T^{\sigma,  0,   \frac{1}{2}+\frac{1}{24}}}\leq C \|u_0\|_{G^{\sigma,  0}}.
\end{equation*}
Moreover, for $\theta \in [0,\frac{1}{4})$, one has $\epsilon(-\theta)=\min(\frac{1}{24}, \frac{-\theta}{6}+\frac{1}{24})\leq \frac{1}{24}$. Consequently, it follows that $u \in X_T^{\sigma,  0,   \frac{1}{2}+\frac{1}{24}} \subset X_T^{\sigma,  0,   \frac{1}{2}+\epsilon(-\theta)}$ and
\begin{equation}\label{eq:estimate u mass}
    \|u\|_{X_T^{\sigma,  0,   \frac{1}{2}+\epsilon(-\theta)}}\leq \|u\|_{X_T^{\sigma,  0,   \frac{1}{2}+\frac{1}{24}}}\leq C \|u_0\|_{G^{\sigma,  0}}. 
\end{equation}
Finally, from \eqref{eq:conserved quantity ZK} and the estimate \eqref{eq:estimate u mass}, we conclude that
\begin{equation*}
\begin{split}
    \sup_{t\in [0, T]} M_\sigma (t) &\leq M_\sigma (0) + C\sigma^{\theta}\|u\|_{X_T^{\sigma,  0,   \frac{1}{2}+\epsilon(-\theta)}}^{3}\\
    &\leq M_\sigma (0) + C\sigma^{\theta} M_\sigma (0)^{\frac{3}{2}}.
    \end{split}
\end{equation*}
\end{proof}

\begin{remark}\label{remark:faixa}
    The refined version of the bilinear estimate proved in Lemma \ref{lemma:estimates partial ZK} played a crucial role in the proof of Corollary \ref{cor: almost conserved quantity ZK}.  The increasing nature of $\epsilon = \epsilon(s)$ as a function of $s$ played a vital role to obtain the estimate \eqref{eq:estimate u mass}. If there was no any information that for any $\theta \in [0,\frac{1}{4})$, $\epsilon(-\theta)\leq \epsilon(0)$, then it would not have been possible to guarantee that
    \begin{equation*}
        \|u\|_{X_T^{\sigma,  0,   \frac{1}{2}+\epsilon(-\theta)}}\leq \|u\|_{X_T^{\sigma,  0,   \frac{1}{2}+\epsilon(0)}}.
    \end{equation*}
\end{remark}

\subsection{Almost conserved quantity at $H^1$-level of Sobolev regularity} 

In this subsection, we will find a growth estimate for $E_\sigma (t)$ defined in \eqref{eq:almost conserved mZK} and consequently prove that it is an almost conserved quantity at the $H^1$-level of Sobolev regularity. This will allow us to extend the local solution to the mZK equation globally in time and obtain a  lower bound for the evolution of the radius of analyticity $\sigma(t)$ as $t\to \infty$. 

Recall that $E_\sigma(t)$ is given by
\begin{equation}\label{en-m}
E_\sigma (t)= \|u(t)\|^{2}_{G^{\sigma,1}}-\int \partial_x (e^{\sigma |D_{x,y}|}u)\partial_y (e^{\sigma |D_{x,y}|}u)dxdy-  \frac{\mu a}{2}\|e^{\sigma |D_{x,y}|}u\|_{L_{x,y}^4}^4.
\end{equation}

Differentiating \eqref{en-m} with respect to $t$, we obtain
\begin{equation}\label{eq:partial E}
\begin{split}
        \frac{d}{dt}(E_\sigma(t))=& 2\int U \partial_tUdxdy+2\int \partial_x U \partial_x\partial_t Udxdy+2\int \partial_y U \partial_y\partial_t U dxdy-\\
        & -\int \partial_x \partial_t U\partial_y Udxdy-\int \partial_x U\partial_y\partial_t Udxdy- 2\mu a\int U^{3}\partial_t U dxdy.
\end{split}
    \end{equation}
    Applying the operator $e^{\sigma|D_{x,y}|}$ to the gZK equation \eqref{eq:new gZK} with $k=2$, we get
    \begin{equation}\label{eq:op in mZK}
        \partial_t U +(\partial_x^3+\partial_y^3)U+\mu a(\partial_x+\partial_y) U^3=G(U),
        \end{equation}
        where $G(U)$ is given by
\begin{equation}\label{eq:G}
G(U)=\mu a (\partial_x+\partial_y)\left[U^{3}-e^{\sigma |D_{x,y}|}\big((e^{-\sigma |D_{x,y}|}U)^{3}\big)\right].
\end{equation}       

Now, using \eqref{eq:op in mZK} in each term of \eqref{eq:partial E}, one has
\begin{equation*}
    \begin{split}
            \int U\partial_tUdxdy=&-\int U\partial_x^3Udxdy-\int U\partial_y^3Udxdy-\mu a \int U \partial_xU^{3}dxdy-\\
            &-\mu a\int U \partial_yU^{3}dxdy+\int UG(U)dxdy,\\
           \int \partial_xU \partial_x\partial_tUdxdy=&-\int\partial_xU\partial_x^4Udxdy-\int \partial_x U\partial_x\partial_y^3Udxdy- \mu a\int \partial_xU\partial_x^2U^3dxdy-\\
           &-\mu a\int \partial_xU\partial_x\partial_yU^3dxdy+\int \partial_xU\partial_xG(U)dxdy,\\
           \int \partial_yU \partial_y\partial_tUdxdy=&-\int\partial_yU\partial_y\partial_x^3Udxdy-\int \partial_y U\partial_y^4Udxdy- \mu a\int \partial_yU\partial_y\partial_xU^3dxdy-\\
           &-\mu a\int \partial_yU\partial_y^2U^3dxdy+\int \partial_yU\partial_yG(U)dxdy,\\
           \int \partial_x\partial_tU \partial_y Udxdy=&-\int \partial_x^4U\partial_yUdxdy-\int \partial_x\partial_y^3 U \partial_y Udxdy-\mu a \int \partial_x^2U^3\partial_y Udxdy-\\
           &-\mu a \int \partial_x\partial_yU^3\partial_y Udxdy +\int \partial_xG(U)\partial_yUdxdy,\\           
           \int \partial_xU \partial_y \partial_tUdxdy=&-\int \partial_xU\partial_y\partial_x^3Udxdy-\int \partial_x U \partial_y^4 Udxdy-\mu a \int \partial_xU\partial_y\partial_x U^3dxdy-\\
           &-\mu a \int \partial_xU\partial_y^2 U^3dxdy +\int \partial_xU\partial_yG(U)dxdy,\\                      
            \int U^{3}\partial_t Udxdy=&-\int U^{3}\partial_x^3Udxdy-\int U^{3}\partial_y^3Udxdy-\mu a\int U^{3}\partial_xU^3dxdy-\\
            &-\mu a\int U^{3}\partial_yU^3dxdy+\int U^{3}G(U)dxdy.
    \end{split}
\end{equation*}

As in the previous subsection, we can assume that $U$ and all its partial derivatives tend to zero as $|(x,y)| \rightarrow \infty$. With this consideration, it follows from integration by parts that

\begin{align}
        \int U\partial_tUdxdy=&\int UG(U)dxdy, \label{eq:1a}\\
        \int \partial_xU \partial_x\partial_tUdxdy =&-\mu a\int U^{3}\partial_x^3U dxdy-\mu a\int \partial_yU\partial_x^2U^3 dxdy+\int \partial_xU\partial_xG(U)dxdy,\label{eq:1b}\\
        \int \partial_yU \partial_y\partial_tUdxdy =&-\mu a\int \partial_xU\partial_y^2U^3 dxdy-\mu a\int U^{3}\partial_y^3U dxdy+\int \partial_yU\partial_yG(U)dxdy,\label{eq:2b}\\
        \int \partial_x\partial_tU \partial_y Udxdy=&-\mu a \int \partial_x^2U^3\partial_y Udxdy-\mu a \int \partial_y^2U^3\partial_x Udxdy +\int \partial_xG(U)\partial_yUdxdy,       \label{eq:1c}\\
         \int \partial_xU \partial_y \partial_tUdxdy=&-\mu a \int \partial_yU\partial_x^2 U^3dxdy-\mu a \int \partial_xU\partial_y^2 U^3dxdy +\int \partial_xU\partial_yG(U)dxdy,                     \label{eq:2c} \\ 
        \int U^{3}\partial_t Udxdy=&-\int U^{3}\partial_x^3Udxdy-\int U^{3}\partial_y^3Udxdy+\int U^{3}G(U)dxdy \label{eq:1d}.
\end{align}

Now, inserting the identities \eqref{eq:1a}, \eqref{eq:1b}, \eqref{eq:2b}, \eqref{eq:1c}, \eqref{eq:2c} and \eqref{eq:1d} in \eqref{eq:partial E}, we get
\begin{equation}\label{eq:partial E new}
\begin{split}
    \frac{d}{dt}E_\sigma(t)=&2\int UG(U)dxdy + 2\int \partial_xU\partial_xG(U)dxdy+2\int \partial_yU\partial_yG(U)dxdy -\\
     &- \int \partial_yU\partial_xG(U)dxdy- \int \partial_xU\partial_yG(U)dxdy-2 \mu a\int U^{3}G(U)dxdy.
\end{split}
\end{equation}

Integrating \eqref{eq:partial E new} in time over $[0, t']$ for $0<t'\leq T$, we obtain
\begin{equation}\label{eq:estimate E}
    E_\sigma(t') =E_\sigma(0)+R_\sigma(t'),
\end{equation}
where
\begin{equation}\label{eq:R mZK}
\begin{split}
    R_\sigma(t')=& \,\, 2\int \int \chi_{[0, t']}UG(U)dxdydt + 2\int\int \chi_{[0, t']}\partial_xU\partial_xG(U)dxdydt +\\
    &+ 2\int\int \chi_{[0, t']}\partial_yU\partial_yG(U)dxdydt-\int\int \chi_{[0, t']}\partial_yU\partial_xG(U)dxdydt-\\
    &-\int\int \chi_{[0, t']}\partial_xU\partial_yG(U)dxdydt - 2\mu a\int\int \chi_{[0, t']}U^{3}G(U)dxdydt.
    \end{split} 
\end{equation}

In sequel, we find estimates for $G(U)$ in the Bourgain's space norm.

\begin{lemma}\label{lemma: estimates G}
Let $G$ be as defined in \eqref{eq:G} and $\sigma >0$. Then, for any $b=\frac{1}{2}+\epsilon$, $0<\epsilon \ll 1$, and for all $\alpha \in \left[0,\frac{3}{4}\right]$,
\begin{equation}\label{eq: estimate G in L2}
\|G(U)\|_{L_{t,x,y}^{2}} \leq C \sigma^{\alpha} \|U\|_{X^{1,   b}}^{3},
\end{equation}
\begin{equation}\label{eq: estimate G in X}
\|\partial_xG(U)\|_{X^{0,   b-1}} \leq C \sigma^{\alpha} \|U\|_{X^{1,   b}}^{3}
\end{equation}
for some constant $C>0$ independent on $\sigma$.
\end{lemma}

\begin{proof}
Observe that
\begin{equation}\label{eq: transf G}
|\widehat{G(U)}(\gamma, \tau)| \leq C|\xi+\eta|\int_{*} \big(1-e^{-\sigma(|\gamma_1|+|\gamma_2|+|\gamma_3|-|\gamma|)}\big) |\widehat{U}(\gamma_1,\tau_1)|  \,|\widehat{U}(\gamma_2,\tau_2)| \,|\widehat{U}(\gamma_3,\tau_3)|,
\end{equation}
where $\int_*$ denotes the integral over the set $\gamma=\gamma_1 + \gamma_2 +\gamma_3$ and $\tau=\tau_1 + \tau_2 +\tau_3$.
 
Now, from the inequality \eqref{eq:exp function}, we get
\begin{equation}\label{eq: 1-e mZK}
    1-e^{-\sigma(|\gamma_1|+|\gamma_2|+|\gamma_3|-|\gamma|)} \leq \sigma^{\alpha}(|\gamma_1|+|\gamma_2|+|\gamma_3|-|\gamma|)^{\alpha}.
\end{equation}
Without loss of generality, we can assume that $|\gamma_1| \leq |\gamma_2| \leq |\gamma_3|$. A simple calculation shows that
\begin{equation*}
    |\gamma_1|+|\gamma_2|+|\gamma_3|-|\gamma| \leq 6|\gamma_2|,
\end{equation*}
and consequently, the estimate \eqref{eq: 1-e mZK} yields
\begin{equation}\label{eq: 1-e nova}
     1-e^{-\sigma(|\gamma_1|+|\gamma_2|+|\gamma_3|-|\gamma|)} \leq C\sigma^{\alpha}|\gamma_{2}|^{\alpha}.
\end{equation}
Using \eqref{eq: 1-e nova} in \eqref{eq: transf G}, we have
\begin{equation}\label{eq: tranf G nova}
|\widehat{G(U)}(\gamma, \tau)| \leq C\sigma^{\alpha}|\xi+\eta|\int_{*} |\gamma_2|^{\alpha}|\widehat{U}(\gamma_1,\tau_1)|\,|\widehat{U}(\gamma_2,\tau_2)|\,|\widehat{U}(\gamma_3,\tau_3)|.
\end{equation}
Moreover, one has
\begin{equation}\label{eq:4.19}
    |\xi+\eta\|\gamma_2|^{\alpha}\leq|\gamma\|\gamma_2|^{\alpha}\leq 3|\gamma_3\|\gamma_2|^{\alpha}.
\end{equation}
Now, using \eqref{eq:4.19} in \eqref{eq: tranf G nova} and an use of Plancherel's Theorem implies that
\begin{equation}\label{eq-420}
\begin{split}
\|G(U)\|_{L_{t,x,y}^{2}}&\leq C \sigma^{\alpha} \left\| \int_{*} |\widehat{U}(\gamma_1,\tau_1)|\,|\widehat{D_{x,y}^{\alpha} U}(\gamma_{2},\tau_{2})|\,|\widehat{D_{x,y} U}(\gamma_{3},\tau_{3})|\,\right\|_{L_{\xi,\eta,\tau}^2}\\
    &=C\sigma^{\alpha} \|w_1 w_2 w_{3}\|_{L_{t,x,y}^{2}},
\end{split}
\end{equation}
where $\widehat{w_1}(\gamma,   \gamma)=|\widehat{U}(\gamma,\tau)|$, $\widehat{w_2}(\gamma,   \tau)=|\widehat{D_{x,y}^{\alpha}U}(\gamma,\tau)|$ and $\widehat{w_3}(\gamma,   \tau)=|\widehat{D_{x,y}U}(\gamma,\tau)|$. 

By applying Lemma \ref{lemma: L2 norm of product mZK},  it follows from \eqref{eq-420} that
\begin{equation*}
\begin{split}
    \|G(U)\|_{L_{x,y}^2 L_t^2} &\leq C{\sigma}^{\alpha} \|w_1\|_{X^{\frac{1}{2}+,b}}\|w_{2}\|_{X^{0,b}}\|w_3\|_{X^{0,b}}\\
    &=C{\sigma}^{\alpha} \|U\|_{X^{\frac{1}{2}+,b}}\|D_{x,y}^{\alpha}U\|_{X^{0,b}}\|D_{x,y}U\|_{X^{0,b}}\\
    &\leq C{\sigma}^{\alpha} \|U\|_{X^{1,b}}^{3}.
\end{split}
\end{equation*}
This completes the proof of \eqref{eq: estimate G in L2}.

To prove \eqref{eq: estimate G in X}, first observe that for $0 \leq j \leq 1$, one has $\langle \gamma \rangle^{-j} \leq C|\gamma|^{-j}$ for all $\gamma \neq 0$. Using this fact, we obtain
\begin{equation}\label{eq: partial G}
\begin{split}
\|\partial_xG(U)\|_{X^{0,  b-1}} &=\left\|\langle \tau -\xi^3-\eta^3\rangle^{b-1} |\xi|\, |\widehat{G(U)}(\gamma,\tau)| \,\right\|_{L_{\tau,\xi,\eta}^{2}}\\
&\leq C\left\|\langle \tau - \xi^3-\eta^3 \rangle^{b-1}\langle\gamma \rangle^j |\gamma|^{1-j}|\widehat{G(U)}(\gamma,\tau)|\,\right\|_{L_{\tau,\xi,\eta}^{2}}.
\end{split}
\end{equation}
Assuming $|\gamma_1|\leq|\gamma_2| \leq |\gamma_3|$, one has $|\gamma|^{1-j} \leq 3^{1-j}|\gamma_{3}|^{1-j}$ and using  \eqref{eq: tranf G nova}, the estimate \eqref{eq: partial G} yields
\begin{equation}\label{eq:4.23}
    \begin{split}
\|&\partial_xG(U)\|_{X^{0,  b-1}}\\  
&\leq 
C\sigma^{\alpha}  \left\| \langle \tau - \xi^3-\eta^3 \rangle^{b-1}\langle\gamma \rangle^j |\gamma|^{1-j}  |\xi+\eta|\!\!\int_{*}\!\! |\gamma_2|^{\alpha}|\widehat{U}(\gamma_1,\tau_1)|\, |\widehat{U}(\gamma_2,\tau_2)|\,|\widehat{U}(\gamma_3,\tau_3)|\,\right\|_{L_{\tau,\xi,\eta}^{2}}\\
&\leq C\sigma^{\alpha}  \left\| \langle \tau - \xi^3-\eta^3 \rangle^{b-1}\langle\gamma \rangle^j  |\xi+\eta|\!\!\int_{*}\!\! |\gamma_{2}|^{\alpha}|\gamma_{3}|^{1-j}|\widehat{U}(\gamma_1,\tau_1)| \,|\widehat{U}(\gamma_2,\tau_2)|\,|\widehat{U}(\gamma_3,\tau_3)|\,\right\|_{L_{\tau,\xi,\eta}^{2}}\\
&=C\sigma^{\alpha} \left\|(\partial_x+\partial_y)(v_1 v_2 v_3)\right\|_{X^{j,   b-1}},
\end{split}
\end{equation}
for $v_1, v_2, v_3$ defined by $\widehat{v_1}(\gamma,   \tau)=|\widehat{U}(\gamma,\tau)|$, $\widehat{v_2}(\gamma,   \tau)=|\widehat{D_{x,y}^{\alpha}U}(\gamma,\tau)|$ and $\widehat{v_3}(\gamma,   \tau)=|\widehat{D_{x,y}^{1-j}U}(\gamma,\tau)|$.

Considering $j=\frac{1}{4}$, we can use the estimate \eqref{eq:partial estimates mZK} with $b=\frac{1}{2}+\epsilon$, to obtain from \eqref{eq:4.23} that
\begin{equation}\label{eq: partial G final}
\begin{split}
  \|\partial_xG(U)\|_{X^{0,  b-1}} &\leq C\sigma^{\alpha}\|v_1\|_{X^{\frac{1}{4}, b}}\|v_2\|_{X^{\frac{1}{4}, b}}\|v_{3}\|_{X^{\frac{1}{4}, b}}\\    &=C\sigma^{\alpha}\|U\|_{X^{\frac{1}{4}, b}} \|D_{x,y}^{\alpha}U\|_{X^{\frac{1}{4}, b}}\|D_{x,y}^{1-j}U\|_{X^{\frac{1}{4}, b}}\\  &\leq C\sigma^{\alpha}\|U\|_{X^{1, b}} \|U\|_{X^{\frac{1}{4}+\alpha, b}}\|U\|_{X^{1, b}}.
    \end{split}
\end{equation}
For $0\leq \alpha\leq \frac{3}{4}$, the estimate \eqref{eq: partial G final} yields
\begin{equation*}
    \|\partial_xG(U)\|_{X^{0,  b-1}} \leq C\sigma^{\alpha} \|U\|_{X^{1,b}}^{3},
\end{equation*}
as desired.
\end{proof}

In the following we use the estimate obtained in Lemma \ref{lemma: estimates G} to prove that the quantity $E_{\sigma}(t)$ defined in \eqref{eq:almost conserved mZK} is almost conserved. We start by proving the following estimate.

\begin{proposition}\label{prop:conserved quantity mZK}
    Let $\sigma>0$ and $\alpha \in \left[0,   \frac{3}{4}\right]$. Then, for any $b=\frac{1}{2}+\epsilon$, $0<\epsilon \ll 1$, there exists $C>0$ such that for any solution $u \in X_T^{\sigma,  1,   b}$ to the IVP \eqref{eq:new gZK} with $k=2$ in the interval $[0, T]$, we have
\begin{equation}\label{eq:conserved quantity mZK}
    \sup_{t\in [0, T]} E_\sigma (t) \leq E_\sigma (0) + C\sigma^{\alpha}\|u\|_{X_T^{\sigma,  1,   b}}^{4}(1+\|u\|_{X_T^{\sigma,  1,   b}}^{2}).
\end{equation}
\end{proposition}

\begin{proof}
   In view of the relations \eqref{eq:estimate E} and \eqref{eq:R mZK}, we first find estimates for each term of $|R_\sigma(t')|$ for all $0<t'\leq T$. For the first and the last terms in \eqref{eq:R mZK} we use the Cauchy-Schwarz inequality, Lemmas \ref{lemma:restriction} and \ref{lemma: L2 norm of product mZK} and the estimate \eqref{eq: estimate G in L2} restricted to the time slab and obtain that for any $b=\frac{1}{2}+\epsilon$, $0<\epsilon \ll 1$, there exists $C>0$ such that
\begin{equation}\label{eq:estimate first term}
    \begin{split}
    \left|\int \int \chi_{[0, t']}UG(U)dxdydt\right| &\leq \|\chi_{[0, t']}U\|_{L_{t,x,y}^{2}}\|\chi_{[0, t']}G(U)\|_{L_{t,x,y}^{2}}\\
    &\leq \|U\|_{X_T^{0, 0}}\|G(U)\|_{X_T^{0, 0}}\\
    &\leq C\sigma^{\alpha}\|u\|_{X_T^{ \sigma,   1,   b}}^{4}
    \end{split}
\end{equation}
and 
\begin{equation}\label{eq:estimate third term}
    \begin{split}
    \left|\int \int \chi_{[0, t']}U^{3}G(U)dxdydt\right| &\leq \|\chi_{[0, t']}U^{3}\|_{L_{t,x,y}^{2}}\|\chi_{[0, t']}G(U)\|_{L_{t,x,y}^{2}}\\
    &\leq \|U^{3}\|_{X_T^{0, 0}}\|G(U)\|_{X_T^{0, 0}}\\
    &\leq C\sigma^{\alpha}\|u\|_{X_T^{ \sigma,   1,   b}}^{6},
    \end{split}
\end{equation}
for all $0<t'\leq T$.

For the second term in \eqref{eq:R mZK}, we use the Cauchy-Schwarz inequality, Lemma \ref{lemma:restriction} and estimate \eqref{eq: estimate G in X}, to obtain
\begin{equation}\label{eq:estimate second term}
    \begin{split}
    \left|\int \int \chi_{[0, t']}\partial_xU\partial_xG(U)dxdydt\right|& \leq \|\chi_{[0, t']}\partial_xU\|_{X^{0,   1-b}}\|\chi_{[0, t']}\partial_xG(U)\|_{X^{0,   b-1}}\\
    & \leq \|\partial_xU\|_{X_T^{0,   1-b}}\|\partial_xG(U)\|_{X_T^{0,   b-1}}\\
    &\leq C\sigma^{\alpha}\|u\|_{X_T^{ \sigma,   1,   b}}^{4},
    \end{split}
\end{equation}
 for all $0<t'\leq T$, where $\frac{1}{2}<b<1$ is the same as before. The estimates for the remaining terms in \eqref{eq:R mZK} are derived using a similar approach. Using these estimates and taking the supremum over $0<t'\leq T$ in \eqref{eq:estimate E}, it follows that
\begin{equation*}
    \begin{split}
    \sup_{t'\in [0, T]} E_\sigma (t') &\leq E_\sigma (0) + \sup_{t'\in [0, T]}|R_\sigma(t')|\\
    & \leq E_\sigma (0) + C\sigma^{\alpha}\|u\|_{X_T^{\sigma,  1,   b}}^{4}(1+\|u\|_{X_T^{\sigma,  1,   b}}^{2}),
    \end{split}
\end{equation*}
as required. 
 
\end{proof}

Finally, we prove that $E_{\sigma}(t)$ is an almost conserved quantity at the $H^1$-level of Sobolev regularity as a corollary of the previous proposition.

\begin{corollary}\label{cor: almost conserved quantity mZK}
Let $\sigma>0$ and $\alpha \in \left[0,\frac{3}{4}\right]$. Then, for any $b=\frac{1}{2}+\epsilon$, $0<\epsilon \ll 1$, there exists $C>0$ such that for any solution $u \in X_T^{\sigma,  1,   b}$ to the IVP \eqref{eq:new gZK} with $k=2$ in the defocusing case ($\mu=-1$) given by Theorem \ref{teo:local mZK}, we have
\begin{equation}\label{eq:estimate energy}
    \sup_{t\in [0, T]} E_\sigma (t) \leq E_\sigma (0) + C\sigma^{\alpha} E_\sigma (0)^{2}(1+E_\sigma (0)).
\end{equation}
\end{corollary}
\begin{proof}
    Denoting $U=e^{\sigma|D_{x,y}|}u$, for $\mu=-1$, from \eqref{eq:almost conserved mZK}, we have
    \begin{equation}\label{eq:conserved defocusing}
    \begin{split}
    E_\sigma (0)&\geq \|U(0)\|_{L^2}^2+\int\big[\partial_xU(0)\big]^2dxdy + \int \big[\partial_yU(0)\big]^2dxdy -\int\partial_xU(0)\partial_yU(0)dxdy \\
    &=\|U(0)\|_{L^2}^2+\frac{1}{2}\int\Big(\big[\partial_xU(0)\big]^2+\big[\partial_yU(0)\big]^2\Big) dxdy +\\
    &\qquad \qquad \qquad \qquad + \int \Bigg(\frac{\big[\partial_xU(0)\big]^2+\big[\partial_yU(0)\big]^2}{2}-\partial_xU(0)\partial_yU(0)\Bigg) dxdy.
    \end{split}
\end{equation}
Observe that the last integral in \eqref{eq:conserved defocusing} is nonnegative since
\begin{equation}\label{eq:nonnegative}
\frac{\big[\partial_xU(0)\big]^2+\big[\partial_yU(0)\big]^2}{2}-\partial_xU(0)\partial_yU(0)=\frac{1}{2}\Big[\partial_xU(0)-\partial_yU(0)\Big]^2 \geq 0.
\end{equation}
Therefore, in the view of \eqref{eq:nonnegative}, the estimate  \eqref{eq:conserved defocusing} yields
\begin{equation}\label{eq:conserved defocusing new}
\begin{split}
E_\sigma (0)&\geq \frac{1}{2}\|U(0)\|_{L^2}^2+\frac{1}{2}\int\Big(\big[\partial_xU(0)\big]^2+\big[\partial_yU(0)\big]^2\Big) dxdy\\
&=\frac{1}{2}\|u_0\|_{G^{\sigma,   1}}^2.
\end{split}
\end{equation}
From \eqref{eq:bound for solution mZK teo} and \eqref{eq:conserved defocusing new}, we get
\begin{equation}\label{eq:estimate u energy}
    \|u\|_{X_T^{\sigma,  1,   b}}\leq C\|u_0\|_{G^{\sigma,   1}} \leq C E_\sigma (0)^{\frac{1}{2}}.
\end{equation}
Finally, an use of \eqref{eq:estimate u energy} in \eqref{eq:conserved quantity mZK} yields the  required estimate  \eqref{eq:estimate energy}.
\end{proof}

\begin{remark}
    For the solutions to the IVP \eqref{eq:new gZK} with $k=2$ in the focusing case ($\mu=1$), from \eqref{eq:almost conserved mZK} one has
    \begin{equation*}
        E_\sigma (0)\leq \|U(0)\|_{L^2}^2+\int \big[\partial_xU(0)\big]^2dxdy + \int \big[\partial_yU(0)\big]^2dxdy -\int\partial_xU(0)\partial_yU(0)dxdy.
    \end{equation*}
This estimate cannot be used to obtain an estimate of the form \eqref{eq:estimate energy} which plays a crucial role in the argument used to prove Theorem \ref{teo:global mZK}. For this reason, we only obtain the lower bound for the evolution of the radius of analyticity for the defocusing mZK equation.
\end{remark}

\section{Global Analytic Solution - Proof of Theorems \ref{teo:global ZK} and \ref{teo:global mZK}}\label{sec:5}

This section is devoted to providing proofs of the global results stated in Theorems \ref{teo:global ZK} and \ref{teo:global mZK}. The ideia of proof of both results is similar using the conserved quantities \eqref{eq:conserved quantity ZK} and \eqref{eq:conserved quantity mZK}. For the sake of completeness, we present a detailed proof of Theorem \ref{teo:global ZK} and provide some hints for Theorem \ref{teo:global mZK}.

\begin{proof}[Proof of Theorem \ref{teo:global ZK}]
Fix $k=1$, $\sigma_0>0$, $s>-\frac{1}{4}$ and $u_0 \in G^{\sigma_0,   s}(\R^2)$. Moreover, let $\theta \in \left[0, \frac{1}{4}\right)$ and $\epsilon>0$ be as in Corollary \ref{cor: almost conserved quantity ZK}. By the invariance of the ZK equation under reflection $(t, x) \rightarrow (-t, -x)$, it suffices to consider $t \geq 0$. With this consideration, we will prove that the local solution $u$ given by Theorem \ref{teo:local ZK} can be extended to any time interval $[0, T]$ and satisfies
\begin{equation*}
    u \in C([0, T]:G^{\sigma(T), s}) \quad \text{for all } T>0,
\end{equation*}
where
\begin{equation}\label{eq:sigma T}
\sigma(T)\geq cT^{-\frac{1}{\theta}}
\end{equation}
and $c>0$ is a constant depending on $\|u_0\|_{G^{\sigma_0,   s}}, \, \sigma_0, \, s$ and $\theta$.

By Theorem \ref{teo:local ZK}, there exists a maximal time $T^*=T^*(\|u_0\|_{G^{\sigma_0,   s}},  \sigma_0,  s) \in (0,  \infty]$ such that
\begin{equation*}
   u \in C([0, T^*):G^{\sigma_0, s}).
\end{equation*}
If $T^*=\infty$, we are done. So, we assume that $T^*<\infty$ and in this case it remains to prove 
\begin{equation*}
    u \in C([0, T]:G^{\sigma(T), s}) \quad \text{for all } T\geq T^*.
\end{equation*}
If we prove this in the case $s=0$ then the general case essentially reduces to $s=0$ using the inclusion \eqref{eq:Gegrey embedding}. For more details we refer to the work in \cite{SS}.

Assume $s=0$ and let $u \in X_{T_0}^{\sigma_0,  0,   \frac{1}{2}+\frac{1}{24}}$ be the local solution to the IVP \eqref{eq:new gZK} with $k=1$ given by Theorem \ref{teo:local ZK}. We have $M_{\sigma_0}(0) = \|u_0\|_{G^{\sigma_0, 0}}^2$ and we can choose the lifespan $T_0$ given in \eqref{eq:T0 ZK} as
\begin{equation*}
    T_0=\frac{c_0}{(1+M_{\sigma_0}(0))^d}
\end{equation*}
for appropriate constants $c_0>0$ and $d>1$.

Let $T\geq T^*$. We will show that, for $\sigma>0$ sufficiently small, one can repeatedly apply Theorem \ref{teo:local ZK} and Corollary \ref{cor: almost conserved quantity ZK} with time step
\begin{equation}\label{eq:delta}
    \delta=\frac{c_0}{(1+2M_{\sigma_0}(0))^d}.
\end{equation}
to extend the local solution to the interval $[0,T]$.
Since $\delta\leq T_0 \leq T^* \leq T$, it follows that there exists $n \in \N$ such that $T \in \left[n\delta, (n+1)\delta\right)$ and by induction, we will show that for $j \in \{1, \cdots,  n\}$,
\begin{align}
    \sup_{t \in [0,  j\delta]} M_\sigma (t) &\leq M_{\sigma}(0)+2^{\frac{3}{2}}C\sigma^{\theta} jM_{\sigma_0}(0)^{\frac{3}{2}},\label{eq:induction 1 ZK}\\ 
    \sup_{t \in [0,  j\delta]} M_\sigma (t) &\leq 2 M_{\sigma_0} (0), \label{eq:induction 2 ZK}
\end{align}
under the smallness conditions
\begin{equation}\label{eq:condition 1 ZK}
    \sigma \leq \sigma_0
\end{equation}
and
\begin{equation}\label{eq:condition 2 ZK}
    2^\frac{3}{2}\frac{T}{\delta}C\sigma^{\theta}M_{\sigma_0}(0)^{\frac{1}{2}} \leq 1,
\end{equation}
where $C>0$ is the constant from Corollary \ref{cor: almost conserved quantity ZK}.

In the first step $j=1$, from the local well-posedness result stated in Theorem \ref{teo:local ZK}, we cover the interval $[0,  \delta]$ and by Corollary \ref{cor: almost conserved quantity ZK}, we have
\begin{equation}\label{eq:indiction base}
\begin{split}
    \sup_{t\in [0, \delta]} M_\sigma (t) \leq M_\sigma (0) + C\sigma^{\theta} M_\sigma (0)^{\frac{3}{2}}.
    \end{split}
\end{equation}
Using the conditions \eqref{eq:condition 1 ZK} and \eqref{eq:condition 2 ZK} in \eqref{eq:indiction base} we conclude that
\begin{equation*}
    \begin{split}
    \sup_{t \in [0,  \delta]} M_\sigma (t) &\leq M_{\sigma_0} (0) + C\sigma^{\theta} M_{\sigma_0} (0)^{\frac{3}{2}}\\
    &=M_{\sigma_0} (0)\left(1+C\sigma^{\theta} M_{\sigma_0} (0)^{\frac{1}{2}}\right)\\
    &\leq M_{\sigma_0} (0)\left(1+\frac{T}{\delta}C\sigma^{\theta} M_{\sigma_0} (0)^{\frac{1}{2}}\right)\\
    &\leq 2 M_{\sigma_0} (0),
    \end{split}
\end{equation*}
which is \eqref{eq:induction 2 ZK} for $j=1$.

Since $M_\sigma(\delta)\leq 2M_{\sigma_0}(0)$, it follows that
\begin{equation*}
    \|u(\delta)\|{G^{\sigma,1}}^2\leq M\sigma(\delta)\leq 2M_{\sigma_0}(0) <\infty,
\end{equation*}
and one can apply the local well-posedness result with initial data $u(\delta)$ in place of $u_0$ in order to obtain an extension of the solution $u$ to the interval $[\delta,2\delta]$. Moreover, $u \in C([\delta,2\delta]:G^{\sigma,0})$ and $u$ satisfies the almost conservation law from Corollary \ref{cor: almost conserved quantity ZK} in $[\delta,2\delta]$, that is,
\begin{equation}\label{eq:before j=2}
    \sup_{t \in [\delta,  2\delta]} M_\sigma (t) \leq M_{\sigma}(\delta)+C\sigma^{\theta} M_{\sigma}(\delta)^{\frac{3}{2}}.
\end{equation}
For $n\geq 2$, from \eqref{eq:induction 1 ZK} and \eqref{eq:induction 2 ZK} with $j=1$, it follows from \eqref{eq:before j=2} that for $j=2$
\begin{equation}\label{eq:j=2}
\begin{split}
    \sup_{t \in [\delta,  2\delta]} M_\sigma (t) &\leq M_{\sigma}(\delta)+2^{\frac{3}{2}}C\sigma^{\theta} M_{\sigma_0}(0)^{\frac{3}{2}}\\
    &\leq M_\sigma(0) + 2^{\frac{3}{2}}C\sigma^{\theta}2M_{\sigma_0}^{\frac{3}{2}} (0).
    \end{split}
\end{equation}
Since
\begin{equation*}
    2\leq n\leq \frac{T}{\delta},
\end{equation*}
it follows from the conditions \eqref{eq:condition 1 ZK} and \eqref{eq:condition 2 ZK}, and from \eqref{eq:j=2} that
\begin{equation*}
    \begin{split}
    \sup_{t \in [\delta, 2\delta]}  M_\sigma (t)&\leq M_\sigma(0) + 2^{\frac{3}{2}}C\sigma^{\theta}\frac{T  }{\delta}M_{\sigma_0}^{\frac{3}{2}} (0)\\
    &\leq 2 M_{\sigma_0}(0).
    \end{split}
\end{equation*}

Now, assume that \eqref{eq:induction 1 ZK} and \eqref{eq:induction 2 ZK} hold for some $j \in \{1,\cdots,  n-1\}$. For $j+1$, applying the local well-posedness result with initial data $u(j\delta)$ and the estimate \eqref{eq:estimate mass} from Corollary \ref{cor: almost conserved quantity ZK}, we have
\begin{equation}\label{eq:induction passage}
\begin{split}
\sup_{t \in [j\delta, (j+1)\delta]}M_\sigma (t)&\leq M_\sigma(j\delta) + C\sigma^{\theta} M_\sigma^{\frac{3}{2}} (j\delta)\\
&\leq M_\sigma(j\delta) + 2^{\frac{3}{2}}C\sigma^{\theta} M_{\sigma_0}^{\frac{3}{2}} (0)\\
&\leq M_\sigma(0) + 2^{\frac{3}{2}}C\sigma^{\theta}(j+1)M_{\sigma_0}^{\frac{3}{2}} (0).
\end{split}
\end{equation}
Moreover, since
\begin{equation*}
    j+1 \leq n \leq \frac{T}{\delta} ,
\end{equation*}
it follows from the conditions \eqref{eq:condition 1 ZK} and \eqref{eq:condition 2 ZK} and from \eqref{eq:induction passage} that
\begin{equation*}
    \begin{split}
    \sup_{t \in [j\delta, (j+1)\delta]}  M_\sigma (t)&\leq M_\sigma(0) + 2^{\frac{3}{2}}C\sigma^{\theta}\frac{T  }{\delta}M_{\sigma_0}^{\frac{3}{2}} (0)\\
    &\leq 2 M_{\sigma_0}(0).
    \end{split}
\end{equation*}

Thus, we can extend the local solution $u$ to the interval $[0,(n+1)\delta)$, under the assumptions \eqref{eq:condition 1 ZK} and \eqref{eq:condition 2 ZK}. Since $T\geq T^*$, the condition \eqref{eq:condition 2 ZK} must fail for $\sigma=\sigma_0$ since otherwise we would be able to continue the solution in $G^{\sigma_0,   0}$ beyond the time $T$, contradicting the maximality of $T^*$.  Therefore, there exists some $\sigma \in (0, \sigma_0)$ such that 
\begin{equation}\label{eq:sigma mZK}
    2^{\frac{3}{2}}\frac{T}{\delta}C\sigma^{\theta}M_{\sigma_0}^{\frac{1}{2}}(0) = 1.
\end{equation}
Moreover, $\sigma$ satisfies \eqref{eq:condition 1 ZK} and \eqref{eq:condition 2 ZK} and
\begin{equation*}
    \sigma(T)=\left[\frac{c_0}{2^{\frac{3}{2}}TCM_{\sigma_0}^{\frac{1}{2}}(0)(1+2M_{\sigma_0}(0))^a}\right]^{\frac{1}{\theta}}=: c_1T^{-\frac{1}{\theta}},
\end{equation*}
which gives \eqref{eq:sigma T} if we choose $c\leq c_1$.

From Corollary \ref{cor: almost conserved quantity ZK}, one can consider $\theta \in \left[0,  \frac{1}{4}\right)$. Choosing the maximum value of $\theta$, one has
\begin{equation*}
    \sigma(T) \geq cT^{-(4+\epsilon)}
\end{equation*}
and the proof for $s=0$ is complete. For other values of $s \in \R$, the proof follows using the inclusion \eqref{eq:Gegrey embedding} as described above.
\end{proof}

\begin{proof}[Proof of Theorem \ref{teo:global mZK}] 
The proof of this theorem is similar to that of Theorem \ref{teo:global ZK} and we present a few steps here. Our objective is to prove that the local solution $u$ given by Theorem \ref{teo:local mZK} in the defocusing case ($\mu=-1$) can be extended to any time interval $[0, T]$ and satisfies
\begin{equation*}
    u \in C([0, T]:G^{\sigma(T), s}) \quad \text{for all } T>0,
\end{equation*}
where
\begin{equation}\label{eq:sigma T mZK}
\sigma(T)\geq cT^{-\frac{1}{\alpha}}
\end{equation}
and $c>0$ is a constant depending on $\|u_0\|_{G^{\sigma_0,   s}},  \, \sigma_0, \, s$ and $\alpha$.

In this case, we fix $s=1$ and use \eqref{eq:estimate energy} with $\alpha \in \left[0,\frac{3}{4}\right]$. From \eqref{eq:conserved defocusing new}, we have\break $2E_{\sigma_0}(0) \geq \|u_0\|_{G^{\sigma_0, 1}}^2$ and from Theorem \ref{teo:local mZK} with $s=1$, we can take the lifespan $T_0$ given in \eqref{eq:T0 mZK} as
\begin{equation*}
    T_0=\frac{c_0}{(1+2E_{\sigma_0}(0))^d}
\end{equation*}
for appropriate constants $c_0>0$ and $d>1$. With this in mind, we aim to show that the local solution $u$ can be extended to any time interval $[0,T]$ using repeatedly Theorem \ref{teo:local mZK} and Corollary \ref{cor: almost conserved quantity mZK} with the time step
\begin{equation}\label{eq:delta mZK}
    \delta=\frac{c_0}{(1+4E_{\sigma_0}(0))^d},
\end{equation}
where $T\geq T^*$ and $T^*$ is the maximal time of existence.

For this purpose, since $\delta\leq T_0 \leq T^* \leq T$, it follows that there exists $n \in \N$ such that $T \in \left[n\delta, (n+1)\delta\right)$ and, by induction, it can be shown that for $j \in \{1, \cdots,  n\}$,
\begin{align}
    \sup_{t \in [0,  j\delta]} E_\sigma (t) &\leq E_{\sigma}(0)+2^{3}C\sigma^{\alpha} jE_{\sigma_0}(0)^{2}(1+E_{\sigma_0}(0)),\label{eq:induction 1 mZK}\\ 
    \sup_{t \in [0,  j\delta]} E_\sigma (t) &\leq 2 E_{\sigma_0} (0), \label{eq:induction 2 mZK}
\end{align}
under the smallness conditions
\begin{equation}\label{eq:condition 1 mZK}
    \sigma \leq \sigma_0
\end{equation}
and
\begin{equation}\label{eq:condition 2 mZK}
    2^{3}\frac{T}{\delta}C\sigma^{\alpha}E_{\sigma_0}(0)^2(1+E_{\sigma_0}(0)) \leq 1,
\end{equation}
where $C>0$ is the constant from Corollary \ref{cor: almost conserved quantity mZK}. 

As in the proof of Theorem \ref{teo:global ZK},  from the maximality of $T^*$, one concludes that there exists $\sigma \in (0, \sigma_0)$ such that
\begin{equation}\label{eq:sigma}
    2^{4}\frac{T}{\delta}C\sigma^{\alpha}E_{\sigma_0}(0)(1+E_{\sigma_0}(0)) = 1,
\end{equation}
and hence
\begin{equation*}
    \sigma(T)=\Big[\frac{\delta}{2^{3}TCE_{\sigma_0}(0)(1+E_{\sigma_0}(0))}Big]^{\frac{1}{\alpha}}=: c_1T^{-\frac{1}{\alpha}},
\end{equation*}
which gives \eqref{eq:sigma T mZK} if we choose $c\leq c_1$.

Considering $\alpha=\frac{3}{4}$, which is the maximum value that can be considered in view of Corollary \ref{cor: almost conserved quantity mZK}, one has
\begin{equation*}
    \sigma(T) \geq cT^{-\frac{4}{3}}
\end{equation*}
and the proof for $s=1$ is concluded.
\end{proof}

\subsection*{Acknowledgment}
The first author acknowledges the  support from CAPES, Brazil and the second author acknowledges the grants from CNPq, Brazil (\#307790/2020-7) and FAPESP, Brazil (\#2024/10613-4). The authors would also like to thank the unanimous referees whose comments helped immensely to improve the original manuscript.\\

\end{document}